\documentclass[12pt]{iopart}

\expandafter\let\csname equation*\endcsname\relax
\expandafter\let\csname endequation*\endcsname\relax
\usepackage{amsmath,amssymb, amsthm, ulem, comment}

\usepackage{bm}
\usepackage{color}
\usepackage{iopams}

\usepackage[colorlinks=true, pdfstartview=FitV, linkcolor=blue,citecolor=blue, urlcolor=blue]{hyperref}

\usepackage[notref,notcite,color]{showkeys}

\usepackage[capitalize,nameinlink,noabbrev]{cleveref}
\usepackage{graphics}
\usepackage{graphicx}
\usepackage{subfigure}
\usepackage{multirow}%

\usepackage{amsfonts, mathrsfs, bbm, cjhebrew, gensymb, textcomp, mathtools, dsfont, calligra}

\numberwithin{equation}{section}

\usepackage{relsize}
\usepackage{scalerel}
\usepackage{stackengine, wasysym}
\usepackage{dsfont}
\usepackage[sort,nocompress]{cite}

\numberwithin{equation}{section}

\newtheorem{theorem}{Theorem}[section]
\newtheorem{lemma}{Lemma}[section]
\newtheorem{remark}{Remark}
\newtheorem{corollary}[lemma]{Corollary}

\newcommand{\R}{\mathbb{R}}
\newcommand{\RR}{\mathbb{R}}

\newcommand{\rank}{\mathrm{rank}}
\newcommand{\lam}{\lambda}
\newcommand{\Lam}{\Lambda}
\newcommand{\diag}{\mathrm{diag}}

\newcommand{\lb}{\langle}
\newcommand{\rb}{\rangle}
\newcommand{\bdy}{\partial}
\newcommand{\tri}{\triangle}
\newcommand{\de}{\delta}
\newcommand{\De}{\Delta}

\newcommand{\til}[1]{\tilde{#1}}
\newcommand{\tbdy}{\til{\bdy}}

\newcommand{\tlam}{\til{\lam}}
\newcommand{\tLam}{\tilde{\Lam}}
\newcommand{\tL}{\widetilde{\mathcal{L}}}
\newcommand{\tu}{\til{u}}
\newcommand{\tK}{\widetilde{\mathcal{K}}}
\newcommand{\tW}{\widetilde{\mathcal{W}}}

\newcommand{\tg}{\til{g}}
\newcommand{\PP}{\mathcal{P}}

\newcommand{\II}{\mathcal{I}}

\newcommand{\td}{\til{d}}
\newcommand{\tv}{\til{v}}
\newcommand{\JJ}{\mathcal{J}}
\newcommand{\dw}{\dot{w}}
\newcommand{\dz}{\dot{z}}
\newcommand{\du}{\dot{u}}
\newcommand{\dv}{\dot{v}}
\newcommand{\dtu}{\dot{\tu}}
\newcommand{\dtv}{\dot{\tv}}
\newcommand{\drho}{\dot{\rho}}
\newcommand{\deta}{\dot{\eta}}

\newcommand{\gam}{\gamma}
\newcommand{\Ra}{\operatorname*{Ra}}

\newcommand{\bv}[1]{\mathbf{#1}}
\newcommand{\innerp}[2]{\langle #1, #2 \rangle}

\newcommand{\p}{\partial}
\newcommand{\del}{\nabla}

\newcommand{\add}[1]{#1}

\begin{document}

\title[Relaxation-based schemes for on-the-fly parameter estimation in dissipative systems]{Relaxation-based schemes for on-the-fly parameter estimation in dissipative dynamical systems}

\author{Vincent R. Martinez$^{1,2}$, Jacob Murri$^3$, Jared P. Whitehead$^4$ }

\address{$^1$Department of Mathematics \& Statistics, Hunter College, City University of New York, New York, NY 10065, USA}
\address{$^2$Department of Mathematics, Graduate Center, City University of New York, New York, NY 10016, USA}
\address{$^3$Department of Mathematics, University of California--Los Angeles, Los Angeles, CA 90031, USA}
\address{$^4$Mathematics Department, Brigham Young University, Provo, UT 84602, USA}

\ead{vrmartinez@hunter.cuny.edu}
\ead{jwmurri@math.ucla.edu}
\ead{whitehead@mathematics.byu.edu}

\vspace{10pt}
\begin{indented}
\item[]\today
\end{indented}

\begin{abstract}
\add{This article studies two particular algorithms, a Relaxation Least Squares (RLS) algorithm and a Relaxation Newton Iteration (RNI) scheme , for reconstructing unknown parameters in dissipative dynamical systems. Both algorithms are based on a continuous data assimilation (CDA) algorithm for state reconstruction of A. Azouani, E. Olson, and E.S. Titi \cite{Azouani_Olson_Titi_2014}. Due to the CDA origins of these parameter recovery algorithms, these schemes provide on-the-fly reconstruction, that is, as data is collected, of unknown state and parameters simultaneously. It is shown how both algorithms give way to a robust general framework for simultaneous state and parameter estimation. In particular, we develop a general theory, applicable to a large class of dissipative dynamical systems, which identifies structural and algorithmic conditions under which the proposed algorithms achieve reconstruction of the true parameters. The algorithms are implemented on a high-dimensional two-layer Lorenz 96 model, where the theoretical conditions of the general framework are explicitly verifiable. They are also implemented on the two-dimensional Rayleigh-B\'enard convection system to demonstrate the applicability of the algorithms beyond the finite-dimensional setting. In each case, systematic numerical experiments are carried out probing the efficacy of the proposed algorithms, in addition to the apparent benefits and drawbacks between them.}
\end{abstract}

\maketitle

\section{Introduction}
\add{This central focus of this article is to carry out an in-depth study of two recent algorithms for estimating unknown parameters of dissipative dynamical systems, one introduced by B. Pachev, J.P. Whitehead, and S. McQuarrie in \cite{pachev2022concurrent}, and another by E. Carlson, J. Hudson, and A. Larios \cite{carlson2020parameter}. Both algorithms are closely related, having been based on the same feedback-control  algorithm for continuous data assimilation \cite{Azouani_Olson_Titi_2014}. However, the methodology of their derivations were different: the former rooted in a control-theoretic perspective, and the latter rooted in an ad-hoc heuristic accounting for leading-order errors. In this article, we provide principled derivations of both algorithms and show that the ad-hoc algorithm introduced in \cite{carlson2020parameter} can in fact be viewed as a Newton-type iteration scheme that we call the \textit{relaxation Newton iteration} (RNI) algorithm, while the algorithm introduced in \cite{pachev2022concurrent} and modified in \cite{murri2022} gives way to what we refer to as a \textit{relaxation least squares} (RLS) algorithm. We identify general conditions under which these algorithms for multi-parameter recovery are guaranteed to converge for a large class of finite-dimensional systems. Although we establish a rigorous convergence result only for the finite-dimensional case, we emphasize that both algorithms are applicable to the infinite-dimensional case of systems of partial differential equations. Indeed, we conclude the paper by carrying out a systematic numerical comparison of both the RLS and RNI algorithms in the context of the two-layer Lorenz 96 system, as well as the Rayleigh-B\'enard convection system, which is based on the Boussinesq approximation of the Navier-Stokes equations for buoyancy-driven incompressible fluids.}

Both the RNI and  RLS algorithms developed here are capable of recovering any finite number of parameters that form the coefficients of linear operators so long as the number of observed degrees of freedom is greater than the number of unknown parameters and certain identifiability-type criteria are satisfied. Indeed, the main theoretical contribution of this paper is to provide a complete and rigorous justification of the convergence properties of both algorithms in the context of finite dimensional systems (see \cref{thm:RNI} and \cref{thm:RLS}). 
In this setting, we apply our general theory to the two-layer Lorenz `96 model (see \cite{lorenz1996predictability}
). We verify that the system satisfies the required hypotheses of \cref{thm:RNI} and \cref{thm:RLS} and computationally implement both algorithms for this model. In the two-layer Lorenz `96 model, we can recover up to $40$ unknown parameters and over $200$ unknown state variables simultaneously while observing only $40$ components of the state.

Although our theoretical result is restricted to finite dimensions, this restriction is only pertinent to the rigorous mathematical study of the general framework. We emphasize that in practice it applies equally well to the infinite dimensional setting of models given by partial differential equations (PDEs). To clarify this, we include a numerical study of the effectiveness of the RLS and RNI algorithms applied to the 2D Rayleigh-B\'enard convection (RBC) system \cite{rayleigh1916lix} in order to recover unknown Rayleigh and Prandtl numbers. Here, we are able to demonstrate, at least computationally, that the algorithm performs exceedingly well in this case, recovering both the state and the unknown Rayleigh and Prandtl numbers with only partial observations of the state.  Further investigations of this case along with rigorous justification for this particular PDE will be performed in a future study.

One common misconception is that because numerical simulation of a PDE requires one to discretize its domain, thereby reducing the PDE (in practice) to a finite-dimensional \textit{discretized model}, any algorithms which are theoretically justified in the finite-dimensional setting are also justified for the PDE. While the application of the algorithm to the \textit{discretized model} would certainly be justified by the theory in this scenario, the theory \textit{does not} justify applying the algorithm to the full model, which is an additional step that must be carried out. Furthermore, this step is decidedly non-trivial as it requires one to develop analysis that is ultimately \textit{independent of the system dimension}. Although we believe it is possible to reasonably extend \cref{thm:RLS} and/or \cref{thm:RNI} to infinite dimensional settings, it is beyond the scope of the current work. Remarkably, this analysis can nevertheless be carried out in certain cases; we refer the reader to \cite{martinez2022convergence}, where such an analysis is carried out for a modified version of RNI in the setting of the externally-driven 2D Navier-Stokes equations for incompressible fluids. In future work, the rigorous convergence analysis of RLS will be borne out in the case of 2D RBC. We simply point out that not only is our framework applicable to PDE models, but the proof itself provides a blueprint for mathematically justifying its applicability for a large class of PDE models as well.

\subsection{Connection to Continuous Data Assimilation}
Both the RNI and RLS algorithms are built on continuous data assimilation (CDA) first introduced in \cite{Azouani_Olson_Titi_2014} applied to the 2D Navier-Stokes equations, and further expanded to several different settings in  \cite{BeOlTi2015,BiHuLaPe2018,JoMaTi2017,larios2023application} for example.  The basic idea of CDA is to introduce a feedback control mechanism (commonly referred to as `nudging') that pushes the dynamical system toward the observed state of the system.  The use of nudging as a data assimilation technique is classical in the context of finite-dimensional systems of ordinary differential equations. The innovation in \cite{azouani2014continuous} was to recognize how various observations, especially in the form of local spatial averages or nodal values, for which direct substitution into the state equation may not be possible, could reliably be assimilated into a \textit{partial differential equation} for the purposes of state reconstruction. Rather than relying on direct insertion, in \cite{azouani2014continuous} the observations are suitably interpolated into the phase space of the system and inserted as an exogenous term that serves to drive the state towards that of the observations. Further extension to different observable variables (other than the velocity in 2D turbulence) is carried out in \cite{FaJoJoTi2017,FarhatJollyTiti2015,farhat2016abridged,farhat2016charney} and different types of observation operators and/or nudging mechanisms are studied in \cite{BiBrJo2020,biswas2021mesh,franz2022bleeps,LaPe2018}.

The algorithm subsequently defined by this procedure \cite{azouani2014continuous} is characterized by a coupled system of dynamical evolution equations which we may generally denote by
\begin{align}\label{eq:CDA_truth}
    \frac{\partial u}{\partial t} &= G(u),\\ \label{eq:CDA_nudge}
    \frac{\partial \tu}{\partial t} &= G(\tu) - \mu I_h(\tu-u),
\end{align}
where (\ref{eq:CDA_truth}) governs the evolution of the \add{reference} 
state, that is, the representation of reality, and (\ref{eq:CDA_nudge}) governs the evolution of our modeled state.  The precise form of the operator $G(u)$ defines which system we are considering (such as the Navier-Stokes equations in which case $u$ is a 2 or 3 dimensional vector field and $G$ is a partial differential operator), while the operator $I_h$ represents a projection onto the observable components of the state, where $h$ characterizes the spatial resolution of the observations. In this sense, $I_h(u(t))$ precisely models the portion of the solution that is observed at the instant $t$ and at length scale $h$.  The parameter $\mu$ represents the inverse time scale at which the modeled state relaxes towards the observed state. \add{The thrust of the work \cite{azouani2014continuous} is that as long as sufficiently many length scales of the motion are observed, then $\mu$ can be tuned so that the modeled state in fact relaxes towards the \textit{full, true state}. In other words, it suffices to enforce relaxation of the system towards the observed finite dimensional state in order to recover the full infinite dimensional state.}

In \cite{azouani2014continuous} and \add{many} later works, the dissipative nature of certain dynamical systems is used to prove that the solution $\tu$ of (\ref{eq:CDA_nudge}) will converge exponentially in time to $u$, the solution of (\ref{eq:CDA_truth}).  These results 
\add{impose} certain conditions be met on the relaxation time scale (set by the parameter $\mu$) and the resolution of the observations (described by $h$) and although the arguments are specific to the choice of $G(u)$ (the choice of system under consideration), the methods are relatively easily adapted to other dissipative systems. CDA is further adapted to the case of an imperfect model in \cite{carlson2020parameter,FarhatGlattHoltzMartinezMcQuarrieWhitehead2020} where it is also noted that the inherent error observable in the state, can be used to give updated parameters for the model (see \cite{carlson2020parameter,FarhatLariosMartinezWhitehead2024} as well).

\subsection{Other approaches to parameter recovery/equation discovery}
While data driven methods such as deep neural networks \cite{duch1994neural} and even Physics-Informed Neural Networks (PINNs), see \cite{cai2021heat,cai2021physics,iten2020discovering} for example, have provided valuable recent advances in the simulation and modeling of physical systems, they have not replaced the need for a physically motivated model whose parameter values are interpretable.  For this reason, methods of equation discovery such as the sparse identification of nonlinear dynamical systems (SINDy) \cite{brunton2016discovering} have been developed in the past decade which not only provide an accurate simulation of the system in question, but also provide formulaic expressions that can be used for further extrapolation of the existing data.
Similar approaches to SINDy have been successfully used to identify additional terms in the equations of motion for fluid mechanics that are thought to best model subgrid scale interactions (see \cite{chattopadhyay2023long,jakhar2023learning} for example), and identification of reduced-order models \cite{geelen2023operator,mcquarrie2023nonintrusive}.  Similar questions are addressed via the use of information geometry \cite{mattingly2018maximizing,quinn2022information}, or even using Bayesian approaches \cite{BoGlKr2020,DaSt2016,mark2018bayesian} that produce detailed estimates on the relevant uncertainty but come at a very high computational cost.

Similar questions arise in the identification of reduced order models, and the methods developed there can similarly be referred to as equation discovery or parameter recovery.  See \cite{mcquarrie2021data,guo2022bayesian,kramer2024learning,fries2022lasdi,conti2024multi,franco2023deep} for a few examples of these efforts wherein a combination of statistical and machine learning methods are used to develop reduced order models that can be used at a fraction of the cost of fully developed physically constrained models.  While our discussion in this paper is related to identifying an unknown parameter for which there is a `true value`, we note that our setup could be used to generate optimal parameters of a parameterized operator that is meant to imitate the dynamics of interest in the system of choice. In other words, although rigorous justification is not available for the application of either RLS and/or RNI to reduced order modeling, it does provide a viable alternative so long as the observable data is approximately continuous in time.

A key distinction between the methods mentioned in the preceding paragraph and those studied in this article, is that both the RLS and RNI algorithms estimate the unknown parameters in an \textit{on-the-fly} manner, that is theier values are ascertained concurrent to the recovery of the state of the system without the need of a costly ensemble or set of Monte Carlo simulations. The methods, thus, stand to achieve significant reductions in computational cost. However, because the methods discussed here rely on CDA, they assume nearly continuous observations of the state in time, which may preclude their application to certain situations.  Lastly, we point out that the theoretical and numerical results that we obtain are developed in the absence of noise. Nevertheless, we expect analogous results to  hold (up to the level of the noise) in the presence of stochastic errors; such studies are reserved for a future work.

\subsection{Organization of the Paper} The rest of this paper will proceed as follows.  \cref{sect:algorithms} provides an intuitive derivation for both algorithms as well as rigorous justification in finite dimensions.  \cref{sect:L96} provides the necessary prerequisite rigorous estimates on the two-layer Lorenz 96 model for the convergence to be justified, and presents numerical simulations that indicate how these algorithms work for this model.  \cref{sect:RBC} includes the modification of both algorithms for the 2D Rayleigh-B\'enard system and some numerical results indicating the validity of both algorithms in this case, and \cref{sect:conc} provides some brief conclusions and outlook for future work.

\section{Dynamical parameter recovery via ``Relaxation"}\label{sect:algorithms}

One of the main insights of this article is to identify a principled derivation of the algorithm that was introduced in \cite{carlson2020parameter} for reconstructing unknown parameters in a nonlinear dynamical system to develop an alternative perspective of a similar algorithm, originally studied in \cite{pachev2022concurrent} and developed further in \cite{FarhatLariosMartinezWhitehead2024} in a different setting, that ultimately expand on the domain and scope of their applicability. 

We will first describe a new heuristic derivation of RNI in a general finite dimensional setting for multiple parameters.
The original derivation \cite{carlson2020parameter} of the RNI algorithm was carried out in an ad-hoc manner, dropping terms that were observed to be of higher-order from the results of numerical experiments. Remarkably, the resulting algorithm had uncannily taken on the form of a Newton iteration scheme, though it was not identified as such then, and robustly achieved reconstruction up to machine precision in a dynamically rich regime for the underlying 2D NSE system.
Thus, our first main goal is to present a more principled derivation of the RNI algorithm. Our derivation highlights the role of the sensitivity equation related to the associated state error and makes definitively clear how and why the algorithm is indeed a form of Newton iteration.

We then expand upon the derivation of what is coined in this paper as the RLS algorithm, which was originally introduced in \cite{pachev2022concurrent} in a basis-dependent form and derived by enforcing a form of null-controllability for the state error. The later work \cite{murri2022} observed that a judicious choice of basis can be obtained by solving a least squares problem, which subsequently gave way to improved numerical performance.  In this work, we formalize the observation in \cite{murri2022} and show that the algorithm from \cite{pachev2022concurrent} is, in fact, a dynamical form of the least-squares method that exploits observations via a Newton relaxation, i.e., nudging, thus justifying its namesake.
Through these new perspectives, we develop a rigorously justifiable general framework for applying both RNI and RLS (see \cref{thm:RNI} and \cref{thm:RLS}) that simultaneously clarifies the relation between the two algorithms.

\subsection{Standing Assumptions} For the remainder of the section, we will assume the following: the reference state is governed by the system
    \begin{align}\label{eq:reference_system}
        \frac{d u}{d t} = \sum_{k=1}^p\lam_k L_k u + F(u),
    \end{align}
where $\lam_k\in\R$, for $k=1,\dots, p$, and $p\geq 1$, where we assume \textit{all variables and parameters have been non-dimensionalized}. For convenience, we restrict ourselves to the setting of finite-dimensional systems, where the solution $u$ satisfies $u:[0,\infty)\rightarrow \mathbb{R}^d$ with $d\geq1$. We moreover assume that $p\leq d$. 

Let us also assume that each of the $L_k:\mathbb{R}^d\rightarrow \mathbb{R}^d$ are linear operators and $F:\mathbb{R}^d\rightarrow\mathbb{R}^d$ is a locally Lipschitz vector field. Note that this finite-dimensional setting is appropriate for the finite-dimensional discretizations of PDEs with polynomial nonlinearity $F$. We will assume in all that follows that given any $\Lam=(\lam_1,\dots,\lam_p)$, there exists a bounded open set $\PP(\Lam)\subset\RR^p$ containing $\Lam$ such that, for all $\tLam\in\PP$, the system \eqref{eq:reference_system} corresponding to $\tLam$ is globally well-posed and dissipative, so that it possesses a well-defined semigroup on $[0,\infty)$ and a global attractor for its dynamics exists. \add{Moreover, the absorbing ball corresponding to \eqref{eq:reference_system} has radius which is uniform on compact subsets of $\PP(\Lam)$. We refer to the set $\PP(\Lam)$ as the \textit{set of admissible parameters}.} 

Observations of \eqref{eq:reference_system} are \add{assumed to be} given \add{in the form of a} continuous time series $\{I_hu(t;u_0)\}_{t\geq0}$, where $I_h$ is a projection operator parameterized by $h$ such that $h^{-1} :=\rank(I_h)\leq d$. Note that if $h=1/d$, then $I_h$ has full rank. Then the problem of parameter estimation for \eqref{eq:reference_system} is the following: 
    \begin{quotation}
    \textbf{Problem.} Determine $\lam_1,\dots,\lam_p$ from knowledge of the time-series $\{I_hu(t;u_0)\}_{t\geq0}$ and the fact that $u$ obeys \eqref{eq:reference_system}, but where the initial condition $u_0$ is unknown.
    \end{quotation}
\add{We will now present the two algorithms, RNI and RLS, to address this problem.}

\subsection{Derivation \& Convergence Analysis of the RNI algorithm}\label{sect:RNI}

The RNI algorithm relies on the use of the so-called nudged system \eqref{eq:CDA_nudge} with  \textit{proxy parameter values}, $\tlam_1,\dots,\tlam_p$:
    \begin{align}\label{eq:nudge_system}
        \frac{d\tu}{dt}=\sum_{k=1}^p\tlam_kL_k\tu+F(\tu)-MI_h\tu+MI_hu,
    \end{align}
    where $M$ is a symmetric, positive definite matrix; it is the generalization of the scalar relaxation parameter $\mu$. In practice, $M$ is typically taken to be a diagonal matrix.
The algorithm in \cite{carlson2020parameter} leverages \eqref{eq:nudge_system} in order to  reconstruct the unobserved portion of the state, $(I-I_h)\tu$, \add{in order to propose a new value of the unknown parameter value. We show in this section how the algorithm proposed in \cite{carlson2020parameter} can be realized as a type of Newton iteration scheme. We then prove a theorem identifying sufficient conditions under which the algorithm can be guaranteed to converge. In a later section, we numerically test the RNI algorithm (see \cref{sect:L96} and \cref{sect:RBC}).} 

\subsubsection{Derivation of the RNI algorithm}

Denote the state error and model error for the $k$th parameter, respectively, by
    \begin{align}\label{def:state:model:error}
        w=\tu-u,\qquad\tri\lam_k=\til{\lam}_k-\lam_k.
    \end{align}
\add{Then $w$ obeys the following evolution equation:}
    \begin{align}\label{eq:state:error}
        \frac{dw}{dt}=\sum_{k=1}^p\left(\tri\lam_k L_k\tu+\lam_k L_kw\right)+\left(F(\tu)-F(u)\right)-MI_hw
    \end{align}
\add{It follows that the observed state error, $I_hw$, is then governed by}
    \begin{align}\label{eq:observed_state_error}
        \frac{d}{dt}I_hw &= \frac{d}{dt} I_h\tu - \frac{d}{dt} I_h u\notag
        \\
        &= \sum_{k=1}^p\left(\tri\lam_kI_hL_k\tu+\lam_kI_hL_kw\right)+\left(I_hF(\tu)-I_hF(u)\right)-I_hMI_hw.
    \end{align}
\add{Further observe that $I_hF(\tu)-   I_hF(u)=I_hF(w+u)-I_hF(u)=:\de_wF_h(u)$, where $\de_w$ denotes the finite difference operator with increment $w$. Thus}
        \begin{align}\label{eq:observed_state_error_mvt}
            \frac{d}{dt}I_hw 
            &= \sum_{k=1}^p\left(\tri\lam_kI_hL_k\tu+   \lam_kI_hL_kw\right)+\de_wF_h(u)-I_hMI_hw.
        \end{align}
 \add{Upon pairing \eqref{eq:observed_state_error_mvt} with $I_hw$, we obtain}
    \begin{align}\label{eq:Pw}
        \|\sqrt{I_hM}I_hw\|^2&=-\frac{1}2\frac{d}{dt}\|I_hw\|^2+\lb \de_wF_h(u),I_hw\rb\notag
        \\
        &\quad\hspace{1.5pt}+\sum_{k=1}^p\left(\tri\lam_k\lb  L_k\tu,I_hw\rb+\lam_k\lb L_kw,I_hw\rb\right).
    \end{align}

\add{For each $k=1,\dots, p$, consider the mapping}
    \begin{align}\label{def:Ej}
         E_k(\tLam)=\frac{1}2\|\sqrt{I_hM}I_hw(\tLam)\|^2.
    \end{align}
Then let $E$ denote the vector
    \begin{align}\label{def:E}
        E(\tLam)=(E_1(\tLam),\dots, E_p(\tLam)).
    \end{align}
Note that $E_i=E_j$, for all $i,j=1,\dots, p$; i.e. $E$ is a constant vector (we represent $E(\tLam)$ in this way to be consistent with the presentation of the RLS algorithm in \cref{sect:RLS} below). \add{We propose to approximate the unknown parameter vector, $\Lam$, by approximating zeroes of $E(\tLam)$. Although, the true parameter vector, $\Lam$, may not necessarily be a zero of $E$, the key insight is that as observations are continuously made, the zero set of $E$ will eventually collapse into a singleton set consisting of the true parameter $\Lam$.}

\add{For each $k=1,\dots, p$, denote the $k$-th partial with respect to $\tLam$ by $\tbdy_k$. 
Then for each $k=1,\dots, p$, the first-order approximation to $E(\tLam)$ is given by}
    \begin{align}\notag
        \eta_{k}=\tbdy_kE_k(\tLam)(\lam_k-\tlam_k)+E_k(\tLam).
    \end{align}
\add{By Newton's method, given a proxy parameter vector, $\tLam$, for the true parameter vector, a new approximation for a zero of $E(\tLam)$ can be found by solving the linear equation, i.e., set $\eta=(\eta_1,\dots, \eta_p)=0$:}
    \begin{align}\label{eq:approx:zero}
        \diag(\tbdy_1E_1(\tLam),\dots,\tbdy_pE_p(\tLam))(\tLam-\Lam)={E(\tLam)}.
    \end{align}
Note the similarity between this approach and the optimization framework introduced in \cite{newey2024}.
    
\add{Now observe that by direct calculation in \eqref{def:Ej}, we have}
    \begin{align}\label{eq:tbdyjE}
        \tbdy_kE_k(\tLam)&=\lb I_hMI_hw(\tLam),\tbdy_kI_hw(\tLam)\rb
    \end{align}
\add{Thus, upon applying the derivative, $\tbdy_k$, in \eqref{eq:Pw}, and making use of the orthogonality of the projection $I_h$, we obtain}
    \begin{align}
        &2\tbdy_kE_k
        =-\lb\bdy_tI_hw,\tbdy_kw\rb-\lb I_hw,\bdy_t\tbdy_kw\rb\notag
        \\
        &\quad+\lb (\de_wDF_h(u))\tbdy_k(w+u),I_hw\rb+\lb DF_h(u)\tbdy_kw,I_hw\rb+\lb \de_wF_h(u),\tbdy_kI_hw\rb\notag
        \\
        &\quad+\lb L_k\tu,I_hw\rb+\lb \tbdy_kL_kw,I_hw\rb+\lb L_kw,\tbdy_kI_hw\rb\notag
        \\
        &\quad +\sum_{\ell=1}^p\tri\lam_\ell\left(\lb L_\ell\tbdy_k\tu,I_hw\rb+\lb L_\ell\tu,\tbdy_kI_hw\rb\right)\label{eq:sensitivity}
    \end{align}
\add{After passing a transient period in $t$, upon invoking analytic dependence in parameters, if $|\tLam-\Lam|\ll1$, then $w(t,\tLam)\approx0$ and $\tbdy_kw(t,\tLam)\approx0$.  This is similar yet distinct from the approach taken in \cite{newey2024} where large scale asymptotics in the nudging parameter $\mu$ (akin to the eigenvalues of $M$) are used to yield the alogrithmic update.  Since the smallest eigenvalues of $M$, may in general have large magnitude, we deduce that to leading order in the state and model  error, we have the following:}
    \begin{align}\notag
        \tbdy_kE_k&\approx \tL_{kk},\quad \tL_{kj}(\tu):=\frac{\de_{kj}}2\lb L_j\tu,I_hw\rb,
    \end{align}
\add{where $\de_{jk}$ denotes the Kronecker delta. Returning to \eqref{eq:approx:zero}, \eqref{eq:tbdyjE}, and making use of the vector notation, we finally arrive at the following approximate relation:}
    \begin{align}\label{eq:RNI:approx}
        \tL(\Lam-\tLam)\approx -E(\tLam).
    \end{align}
\add{Thus, if $\tL$ is invertible, one finally arrives at}
    \begin{align}\notag
        \Lam\approx \tLam -\tL^{-1}E(\tLam).
    \end{align}
\add{This approximate relation produces the following iteration scheme, which we refer to as the \textit{Relaxation Newton Iteration} scheme:}
    \begin{align}\label{def:RNI}
        \Lam^{(n+1)}=\Lam^{(n)}-(\tL^{(n)})^{-1}E(\Lam^{(n)}), \quad \add{\tL^{(n)}_{kj}:=\tL_{kj}(\tu^{(n)})=\frac{\de_{kj}}2\lb L_{j}\tu^{(n)},I_hw^{(n)}\rb},
    \end{align}
\add{where $\tu^{(n)} = \tu(t_{n+1};\tu^{(n-1)},\Lam^{(n)})$ represents the solution of \eqref{eq:nudge_system} corresponding to choice of parameters $\Lam^{(n)}$, after time $t_{n+1}-t_n$ with initial condition $\tu^{(n-1)}$, $I_hw^{(n)} = I_h\tu^{(n)}-I_hu^{(n)}$, and $\tL^{(n)}$ is defined in \eqref{def:RNI}. Note that the solvability condition for \eqref{eq:RNI:approx} is simply that $\prod_{k=1}^p\lb L_k\tu^{(n)},I_hw^{(n)}\rb\neq0$, for all $n$. Although this may be difficult to verify a priori, this is easy to check in practice, and in the course of its implementation, whenever this condition is violated, one may simply wait to evaluate the formula at a more fortuitous time in the future.}

\add{The algorithm defined by \eqref{def:RNI} is a generalization of the algorithm originally derived in \cite{carlson2022} for estimating three different parameters in the Lorenz 63 model, to the case of arbitrarily many parameters $p$. There, an ad hoc derivation was presented for obtaining \eqref{def:RNI} in the particular setting of the Lorenz 63 model.}
\add{In order to prove convergence of \eqref{def:RNI}}, the key steps that must be justified from the derivation is that the state and model errors become small, i.e. $|\tu-u| \ll 1$ and $|\tLam-\Lam| \ll 1$.
Rigorous justification for this requires careful control of the \add{state and model error and their subsequent interplay in order to validate} that the higher order terms that were dropped \add{in the course of the derivation} are indeed negligible.
This was carried out in the context of the Lorenz 63 equations to reconstruct the three relevant parameters appearing there in \cite{carlson2022}, and in the context of the 2D NSE in order to reconstruct the unknown viscosity \cite{martinez2022convergence}.

\subsubsection{Convergence analysis of the RNI algorithm}\label{sect:converge:RNI}
\add{We now move on to the convergence analysis of the RNI algorithm. In order to state and prove the theorem, it will be useful to set up a few preliminary steps. One important such step is a determining a suitable form for the model error.}

\add{To this end, we let $\tLam=\Lam^{(n)}$ and rewrite }\eqref{eq:Pw} as
    \begin{align}\label{eq:Pw:rewrite}
        E_k(\Lam^{(n)})&=-\frac{1}2\left\lb \frac{d}{dt}I_hw^{(n)},I_hw^{(n)}\right\rb+\frac{1}2\left\lb \de_{w^{(n)}}F_h(u^{(n)}),I_hw^{(n)}\right\rb\notag
        \\
        &\quad\hspace{2pt}+\tL_{kk}(\tu^{(n)})\De\Lam_k+\tL_{kk}(w^{(n)})\Lam_k,
    \end{align}
\add{for all $k=1,\dots, p$. Finally, upon combining \eqref{def:RNI} and \eqref{eq:Pw:rewrite}, we may deduce}
    \begin{align}
        \add{\Lam^{(n+1)}_k-\Lam_k}
        &\add{=\left(\Lam^{(n+1)}_k-\Lam^{(n)}_k\right)+\left(\Lam^{(n)}_k-\Lam_k\right)}\notag
        \\
        &\add{=(\tL(\tu^{(n)}))_{kk}^{-1}\frac{1}2\left\lb \frac{d}{dt}I_hw^{(n)},I_hw^{(n)}\right\rb-\frac{1}2(\tL(\tu^{(n)}))_{kk}^{-1}\left\lb \de_{w^{(n)}}F_h(u^{(n)}),I_hw^{(n)}\right\rb}\notag
        \\
        &\quad \add{-(\tL(\tu^{(n)}))^{-1}_{kk}(\tL(w^{(n)}))_{kk}\Lam_k\notag}.
    \end{align}
Thus
    \begin{align}\label{eq:model_error_RNI}
        \add{\tL^{(n)}\De\Lam^{(n+1)}=\tW^{(n)}\Lam+\tg^{(n)}},
    \end{align}
where
    \begin{align}\label{def:tW_tg_RNI}
        \add{\tW^{(n)}:= -\tL(w^{(n)})\Lam,\quad \tg^{(n)}_k:=\frac{1}2\left\lb \frac{d}{dt}I_hw^{(n)},I_hw^{(n)}\right\rb-\frac{1}2\left\lb \de_{w^{(n)}}F_h(u^{(n)}),I_hw^{(n)}\right\rb},
    \end{align}
for each $k=1,\dots, p$. Note that $\tg^{(n)}$ is a constant vector, that is, $\tg^{(n)}_k=\tg^{(n)}_\ell$, for all $k,\ell=1,\dots, p$. Finally, given $t_n>0$, we define the quantity
    \begin{align}\label{def:Un}
        U_n=\sup_{t\geq t_n}\|u(t;u_0)\|<\infty.
    \end{align}
Note that $U_n\leq U_{n+1}$, for all $n\geq0$.

\begin{theorem}\label{thm:RNI}
Suppose that there exists $t'$ such that
    \begin{align}
        \sup_{t\geq t'}\left(\mu^{p}\|\tu(t; \tu_0, \tLam) - u(t; u_0)\|+\mu^{q}\left\|\frac{d}{dt}\left(I_h\tu(t; \tu_0, \tLam) - I_hu(t; u_0)\right)\right\|\right)
        &\leq C \|\tLam-\Lam\|\label{eq:model_error_estimate_RNI}
    \end{align} 
where $\mu$ denotes the smallest eigenvalue of $M$, for some positive constant $C$ and exponents $p,q$, independent of $h$.
Let $\tu^{(-1)}\in\R^d$ denote an estimate of the state at the initial time $t_0=0$. Suppose the initial parameter estimate $\Lam^{(0)}\in\PP(\Lam)$ satisfies
    \begin{align}\label{est:RNI_prior}
        \|\De\Lam^{(0)}\|\leq \rho_0,
    \end{align}
for some $\rho_0>0$. There exists an increasing sequence of times $\{t_n\}_n$, along which \eqref{eq:model_error_estimate_RNI} holds, and positive constants $\{C_n\}_n$ such that 
    \begin{align}\label{eq:parameter_error_bound_RNI}   
         \|\Delta \Lam^{(n+1)}\| \leq D_*\frac{C_n\|I_hw^{(n)}\|\|(\tL^{(n)})^{-1}\|}{\mu_n^*}\|\Delta\Lam^{(n)}\|,
    \end{align}
for all $n\geq1$, for some positive constant $D_*$, where $\mu_n^*:=\min\{\mu_n^p,\mu_n^q\}$, provided that at each stage $\ell=0,\dots, n$, $\mu_\ell$ is chosen such that
    \begin{align}\label{cond:Lam_RNI}
        \Lam^{(\ell+1)}\in\PP(\Lam),\quad \mu_n^p\geq C_0U_{n+1}\rho_0,
    \end{align}
and that
   \begin{align}\label{cond:delta_RNI}
        0< \de_n:=1-D_*\frac{C_n\|I_hw^{(n)}\|\|(\tL^{(n)})^{-1}\|}{\mu_n^*}<1.
    \end{align}
Moreover, if $\sum_n\de_n=\infty$, then 
    \begin{align}\label{eq:converge_RNI}
        \lim_{n\rightarrow\infty}|\Delta\Lam^{(n)}|=0.
    \end{align}
\end{theorem}

\begin{proof}
\add{Let $\tu^0(t)=\tu(t;\tu^{(-1)},\Lam^{(0)},M^{(0)})$ denote the solution of \eqref{eq:nudge_system} corresponding to $\tLam=\Lam^{(0)}$, initial value $\tu^{(-1)}$, and matrix $M^{(0)}$.}

\add{Now let $t_1$ be large enough that \eqref{eq:model_error_estimate_RNI} holds for $t'=t_1$, $\tu_0=u^{(-1)}$, $\tLam=\Lam^{(0)}$, and that $u(t_1;u_0)$, belongs to the absorbing ball of \eqref{eq:reference_system}, so that $U_1<\infty$. Let $\tu^{(0)}=\tu(t_1;\tu^{(-1)},\Lam^{(0)})$ and $u^{(0)}=u(t_1;u_0)$.} Then \eqref{eq:model_error_estimate_RNI} becomes
    \begin{align}\label{eq:state_error_RNI_0}
        \|w^{(0)}\| \leq \frac{C_0}{\mu_0^{p}} \|\Delta \Lam^{(0)}\|,
    \end{align}
\add{where $\mu_0$ denotes the smallest eigenvalue of $M^{(0)}$}. We choose 
    \begin{align}\label{eq:mu_choice_RNI_0}
        \add{\mu_0^p\geq C_0U_1\rho_0},
    \end{align}
\add{so that}
    \begin{align}\label{eq:w_RNI_0}
        \add{\|w^{(0)}\|\leq U_1}.    
    \end{align}

From \eqref{eq:model_error_RNI}, we have
    \[
        \add{\tL^{(0)} \Delta \Lam^{(1)} =\tW^{(0)}\Lam+\tg^{(0)}}
    \]
\add{First observe that}
    \begin{align}\label{eq:W0_RNI}
        \add{\tW^{(0)} \Lam=\tW^{(0)}\De\Lam^{(0)}+\tW^{(0)}\Lam^{(0)}}.
    \end{align}
\add{Since $I_h$ and $L_j$ are bounded operators, it follows from \eqref{def:tW_tg_RNI}, \eqref{def:RNI} that}
    \[
        \add{\|\tW^{(0)}\| \leq \frac{\|L\|}{2}\|w^{(0)}\|\|I_hw^{(0)}\|,}
    \]
\add{where $w^{(0)}=\tu^{(0)}-u^{(0)}$}. Thus, upon returning to \eqref{eq:W0_RNI}, we obtain
    \begin{align}\label{est:model_error_RNI_1a}
        \|\tW^{(0)}\Lam\|
        &\leq \frac{\|L\|}{2}\left(\rho_0+\|\Lam^{(0)}\|\right)\|w^{(0)}\|\|I_hw^{(0)}\|\notag
        \\
        &\leq \frac{C_0U_1\|L\|}{2\mu_0^p}\left(\rho_0+\|\Lam^{(0)}\|\right)\|I_hw^{(0)}\|\|\De\Lam^{(0)}\|.
    \end{align}

On the other hand, let $C_F^{1}$ denote the local Lipschitz constant of $F$ \add{corresponding to the ball $B(0,2U_1)$}. Then
    \begin{align}
        \|\tg^{(0)}\| &\leq \add{\frac{C_0}{4\mu_0^{q}}\|\De\Lam^{(0)}\|\|I_hw^{(0)}\|+\frac{\|I_h\|}2\|F(w^{(0)}+u^{(0)}) - F(u^{(0)})\|\|I_hw^{(0)}\|}\notag
        \\
        &\leq \add{\frac{C_0}{4\mu_0^{q}}\|\De\Lam^{(0)}\|\|I_hw^{(0)}\|+\frac{\|I_h\|}2C^1_F\|w^{(0)}\|\|I_hw^{(0)}\|}\notag
        \\
        &\leq \frac{C_0}2\left(\frac{1}{2\mu_0^{q}}+\frac{\|I_h\|C^1_FU_1}{\mu_0^p}\right)\|I_hw^{(0)}\|\|\De\Lam^{(0)}\|.\label{est:model_error_RNI_1b}
    \end{align}
Upon combining \eqref{est:model_error_RNI_1a}, \eqref{est:model_error_RNI_1b} in \eqref{eq:W0_RNI}, we obtain
    \begin{align}\notag
        \add{\|\Delta \Lam^{(1)}\|}
        &\add{\leq \|(\tL^{(0)})^{-1}\| \left( \|W^{(0)}\Lam\| + \|\tg^{(0)}\| \right)}\notag
        \\
        &\leq \frac{C_0\|I_hw^{(0)}\|\|(\tL^{(0)})^{-1}\|}{2\min\{\mu_0^p,\mu_0^q\}}\left\{U_1\left[\|L\|\left(\rho_0+\|\Lam^{(0)}\|\right)+\|I_h\|C^1_F\right]+\frac{1}{2}\right\}\|\De\Lam^{(0)}\|.\notag
    \end{align}
    
Let
    \[
        \add{D_* := \frac{1}2\left\{U_1\left[\|L\|\left(\rho_0+\|\Lam^{(0)}\|\right)+\|I_h\|C^1_F\right]+\frac{1}{2}\right\}},
    \]
where
    \[
        \|I\|:=\sup_{h>0}\|I_h\|.
    \]
Then
    \begin{align}\label{est:model_error_RNI_0}
        \add{\|\Delta \Lam^{(1)}\| \leq  D_*\frac{C_0\|I_hw^{(0)}\|\|(\tL^{(0)})^{-1}\|}{\mu_0^*} \|\Delta \Lam^{(0)}\|,}
    \end{align}
\add{where $\mu_0^*=\min\{\mu_0^p,\mu_0^{q}\}$, which we choose such that
    \begin{align}\label{eq:delta_RNI_0}
        1-\de_0=D_*\frac{C_0\|I_hw^{(0)}\|\|(\tL^{(0)})^{-1}\|}{\mu_0^*} <1.
    \end{align}
and, simultaneously, such that $\Lam^{(1)}\in\PP(\Lam)$.} We proceed inductively.

Suppose that \eqref{eq:parameter_error_bound_RNI}, \eqref{eq:delta_RNI_0} and that
    \begin{align}\label{eq:mu_choice_RNI_m}
        \mu_m^p\geq C_mU_{m+1}\rho_0,
    \end{align}
all hold for each $m = 0, \ldots, n-1$, for some $n\geq1$. Now consider the case $m = n$. \add{Let $\tu^{n}(t)=\tu(t;\tu^{(n-1)},\Lam^{(n)},M^{(n)})$ denote the solution of \eqref{eq:nudge_system} corresponding to $\tLam=\Lam^{(n)}$, initial value $\tu^{(n)}$, and matrix $M^{(n)}$.} Let $t_{n+1} > 0$ \add{be} large enough \add{so} that \eqref{eq:model_error_estimate} holds \add{for $t'=t_{n+1}$, $\tu_0=u^{(n-1)}$, $\tLam=\Lam^{(n)}$}. Then 
    \[
        \|w^{(n)}\| \leq \frac{C_n}{\mu_n^p} \|\Delta \Lam^{(n)}\|,
    \]
\add{where $\mu_n$ denotes the smallest eigenvalue of $M_n$.} Upon iterating, we see that
    \begin{align}\label{est:state_error_RNI_n}
            \|w^{(n)}\|\leq \frac{C_n}{\mu_n^p}\left(\prod_{j=1}^{n-1}(1-\de_j)\right)\|\De\Lam^{(0)}\|\leq \frac{C_n\rho_0}{\mu_n^p}
    \end{align}
We assume that $\mu_n$ is chosen large enough such that
    \begin{align}\label{eq:mu_choice_RNI_n}
           \mu_n\geq C_nU_{n+1}\rho_0,
    \end{align}
\add{where $U_j$ is defined as in \eqref{def:Un}.

Using the representation of the model error \eqref{eq:model_error_RNI}, recall that}
    \[
        \tL^{(n)}\Delta \Lam^{(n+1)} = \tW^{(n)} \Lam +\tg^{(n)}.
    \]
\add{With \eqref{eq:mu_choice_RNI_n}, we may now argue as in the base case to obtain}
    \begin{align}
        \add{\|\tW^{(n)}\Lam\|}&\add{\leq \frac{C_nU_1\|L\|}{2\mu_n^p}\left(\rho_0+\|\Lam^{(0)}\|\right)\|I_hw^{(n)}\|\|\De\Lam^{(n)}\|}\notag
        \\
        \add{\|\tg^{(n)}\|}&\add{\leq \frac{C_n\|I_hw^{(n)}\|\|(\tL^{(n)})^{-1}\|}{2\min\{\mu_n^p,\mu_n^q\}}\left\{U_1\left[\|L\|\left(\rho_0+\|\Lam^{(0)}\|\right)+\|I_h\|C^1_F\right]+\frac{1}{2}\right\}\|\De\Lam^{(n)}\|}.\notag
    \end{align}
\add{where we implicitly used the fact that $U_{j}\leq U_1$, for all $j\geq0$, which implies $C_F^{j}\leq C_F^1$}. 
Then 
    \begin{align}
        \add{\|\Delta \Lam^{(n+1)}\|}
        &\add{\leq D_*\frac{C_n\|I_hw^{(n)}\|\|(\tL^{(n)})^{-1}\|}{\mu_n^q}\|\De\Lam^{(n)}\|}.\notag
    \end{align}
\add{We conclude the proof by taking $\mu_n$ such that}
    \begin{align}\label{eq:delta_RNI_n}
        1-\de_n=D_*\frac{C_n\|I_hw^{(n)}\|\|(\tL^{(n)})^{-1}\|}{\mu_n^*} <1,
    \end{align}
and, simultaneously, such that $\Lam^{(n+1)}\in\PP(\Lam)$.
\end{proof}

\begin{remark}
We note that \cref{thm:RNI} addresses the issues of convergence for the algorithm defined by \eqref{def:RNI}. However, it does not justify the derivation presented above; this is a separate issue that is not treated here. Nevertheless, it is interesting that the derivation produces an operational algorithm. Justification for the derivation provided here is partially discussed in \cite{newey2024}, and rigorous justification for this derivation will be forthcoming in a future work.
\end{remark}

\subsection{Derivation and Convergence Analysis of the RLS algorithm}\label{sect:RLS}
The algorithm originally derived in \cite{pachev2022concurrent} and further expanded in \cite{FarhatLariosMartinezWhitehead2024} can be \add{obtained by taking} a \add{rather different perspective} 
to \add{that of the RNI algorithm. Indeed, from \eqref{eq:observed_state_error}, heuristically allowing $p=h^{-1}$, we alternatively see that}
\begin{equation}\notag
    \frac{d}{dt}I_h w = \left(\sum_{k=1}^p \tlam_k I_hL_k \tu + I_h F(\tu) - \frac{d}{dt}I_h u \right) - MI_h w.
\end{equation}
\add{The main idea of RLS is to enforce convergence of the observed state error to zero by selecting $\tlam_k$ such that}
    \begin{align}\label{eq:approx:RLS}
        \sum_{k=1}^p \tlam_k I_hL_k \tu + I_h F(\tu)=\frac{d}{dt}I_hu.
    \end{align}
\add{With this choice of parameters, one consequently enforces the following equation:}
    \begin{align}
        \frac{d}{dt}I_hw+MI_hw=0\notag.
    \end{align}
\add{Thus, if $M$ is positive semi-definite, it follows that $\|I_hw(t)\|\rightarrow0$ as $t\rightarrow\infty$. In other words, the approximating parameters are viewed as \textit{control parameters}, which are chosen in such way as to \textit{enforce asymptotic null controllability of the state error}.}

\add{Choosing $\tLam=(\tlam_1,\dots,\tlam_p)$ to satisfy \eqref{eq:approx:RLS} should, in principle, be possible provided that} $I_h$ \add{is time-independent, that is, the range of length scales that are observed and the form in which they are stored, i.e., as spectral data, local spatial averages, etc., do not change as we observe them}
and provided that our observations are \add{sufficiently smooth in time},
so that $\frac{d}{dt}I_hu$ is also observable or \add{at least sufficiently well-approximated by finite differences. However, since $\tu$ depends on $\tLam$, to solve for $\tLam$ directly in \eqref{eq:approx:RLS} is a delicate procedure.}

\add{Nevertheless, one can rigorously define an iterative algorithm based on this idea. Indeed, upon initializing the algorithm with a guess $\Lam^{(n)}=(\lam_1^{(n)},\dots, \lam_p^{(n)})$, the next approximate value, $\Lam^{(n+1)}$, of $\Lam$ can be generated by imposing that they satisfy the following system:} 
    \begin{equation}\label{eq:RLS_full}
        \sum_{k=1}^p \lambda_k^{(n+1)} I_h L_k \tu^{(n)} + I_h F(\tu^{(n)}) = \frac{d}{dt}I_h u^{(n)},
    \end{equation}
\add{where $\tu^{(n)} = \tu(t_{n+1};\tu^{(n-1)},\Lam^{(n)})$ represents the solution of \eqref{eq:nudge_system} corresponding to the choice of parameters $\Lam^{(n)}$, after time $t_{n+1}-t_n$ with initial condition $\tu^{(n-1)}$.}  
The issue with enforcing \eqref{eq:RLS_full} is
that we will typically have $h^{-1} > p$ (unlike the heuristic introduced above), i.e. there are usually far more observations than there are unknown parameters so that \eqref{eq:RLS_full} describes an over-determined system.

To identify an optimal solution of \eqref{eq:RLS_full}, we instead seek to minimize
\begin{equation}\label{eq:RLS_pose}
    \left\| \sum_{k=1}^p\lambda_k^{(n+1)}I_h L_k \tu^{(n)}+ I_h F(\tu^{(n)}) - \frac{d}{dt}I_h u^{(n)}\right\|^2_2,
\end{equation}
where the norm here is taken over \add{$\add{\ell^2}$}.  This \add{leads to finding a least squares solution to the following equation:}
\begin{equation}\label{def:RLS}
    \add{\tL^{(n)}}\Lam^{(n+1)} = f,\quad \add{\tL^{(n)}_{jk}:=(I_hL_k\tu^{(n)})_j,\quad f:=\frac{d}{dt}I_hu^{(n)} - I_h F(\tu^{(n)})}.
\end{equation}
\add{Recall that $\tL^{(n)}$ is a $d\times p$ matrix, and 
 thus is typically overdetermined.} The least-squares solution is given by
\begin{equation}\label{eq:RLS_sol}
    \Lam^{(n+1)} = \add{\left((\tL^{(n)})^T\tL^{(n)}\right)^{-1} (\tL^{(n)})^T} f,
\end{equation}
Note that this solution exists if and only if the columns of \add{$\tL^{(n)}$} are linearly independent at the update time $t_{n+1}$.  If this is not the case, then we can typically wait until a later time $t_{n+1}^*>t_{n+1}$ wherein the observable parts, \add{$I_hL_k$}, of the operators $L_k$ will be linearly independent of each other.  On the other hand, if there are two such linear operators that have redundant roles such as splitting the viscosity into two indistinguishable components $\nu = \nu_1+\nu_2$ then the columns of $\tL$ will be linearly dependent and the least squares solution \textit{will not} exist no matter what our choice of the evaluation time $t_{n+1}$ is. \add{We emphasize that the condition of} linear independence is \add{formulated only on the \textit{observable part}, $I_hL_k$, of these operators, rather than the operators, $L_k$, themselves.}

\begin{remark}
    We note that \add{although} both the RNI and RLS algorithms \add{share a common departure point} from 
    the evolution equation of the observable error, \add{$I_hw$}, \add{the two algorithms possess some distinct advantages and disadvantages in relation to one another}.  The primary advantages of RLS are that it makes no assumptions on the smallness of the nonlinear terms or the temporal derivative of the observable error and therefore provides a simpler route to proving convergence.  \add{However}, RLS does require differentiability-in-time of the observed state (or at least enough smoothness that a reasonable numerical approximation of that quantity is possible). In contrast, only continuous-time observations are required \add{for} 
    the RNI algorithm. 
    
    \add{The RNI algorithm, on the other hand, possesses a very simple update formula \eqref{def:RNI} that does not require one to solve a least-square minimization problem. Moreover, its solvability condition is easily checkable in practical implementation and does not require one to check the linear independence condition required by least-squares.}
\end{remark}

To prove \add{a convergence result analogous to \cref{thm:RNI} for the RLS algorithm}, we first introduce some notation.  Note that the least squares problem \eqref{eq:RLS_pose} is equivalent to solving the system: 
    \begin{equation}\label{eq:RLS_least_square_explicit}
        \sum_{k=1}^p \tlam_k^{(n+1)} \left\langle I_hL_k\tu^{(n)},I_hL_j\tu^{(n)}\right\rangle = \left\langle \frac{d}{dt}I_h u^{(n)} - I_h F(\tu^{(n)}),I_h L_j\tu^{(n)}\right\rangle,
    \end{equation}
\add{for each $j = 1,\ldots , p$.}
From the original evolution equation, we see that
    \begin{equation}\label{eq:least_squares}
        \left\langle \frac{d}{dt}I_h u^{(n)},I_h L_j\tu^{(n)}\right\rangle = \sum_{k=1}^p \lambda_k \left\langle I_h L_ku^{(n)},I_h L_j \tu^{(n)}\right\rangle + \left\langle I_h F(u^{(n)}),I_h L_j \tu^{(n)}\right\rangle.
    \end{equation}
Substituting this back into \eqref{eq:RLS_least_square_explicit} and using $w^{(n)} = \tu^{(n)} - u^{(n)}$ now leads to
    \begin{align}\label{eq:parameter_error}
        &\sum_{k=1}^p (\tlam_k^{(n+1)} - \lambda_k)\left\langle I_hL_k\tu^{(n)},I_h L_j\tu^{(n)}\right\rangle\notag
        \\
        &= \sum_{k=1}^p \lambda_k \left\langle I_hL_k w^{(n)},I_h L_j \tu^{(n)}\right\rangle-\left\langle \de_{w^{(n)}}F_h(u^{(n)}),I_hL_j\tu^{(n)}\right\rangle,
    \end{align}
which is a linear equation \add{representing} the error in the parameter estimate. \add{We refer to \eqref{eq:parameter_error} as the \textit{model error}}.

Now define the following $p\times p$ matrix $\tW^{(n)}$ and $p\times 1$ vector $\tg^{(n)}$ as:
\begin{equation}\label{def:tK_tW}
    \tK^{(n)}_{jk}=\left\langle I_hL_k\tu^{(n)},I_hL_j\tu^{(n)}\right\rangle,\quad\tW_{jk}^{(n)} = \left\langle I_hL_kw^{(n)},I_hL_j\tu^{(n)}\right\rangle,
\end{equation}
and
    \begin{align}\label{def:tf}
        \tg_j^{(n)} = -\left\langle \de_{w^{(n)}}F_h(u^{(n)}),I_hL_j\tu^{(n)}\right\rangle.
    \end{align}
This allows one to rewrite the \add{model error} as
\begin{align}\label{eq:model_error_RLS}
    \tilde{\mathcal{K}}^{(n)} \Delta \Lam^{(n+1)} = \tW^{(n)}\Lam + \tg^{(n)},
\end{align}
where $\Delta \Lam^{(n+1)} = \tLam^{(n+1)}-\Lam$. Note that $\tK=\tL^T\tL$. Before we state the main theorem, we will also once again make use of the notation \eqref{def:Un}.

\begin{theorem}\label{thm:RLS}
Suppose that there exists $t'$ such that
    \begin{align}\label{eq:model_error_estimate}
        \sup_{t\geq t'}\mu^p|\tu(t; \tu_0, \tLam) - u(t; u_0)|
        \leq C|\Delta\Lam|,
    \end{align} 
where $\mu$ denotes the smallest eigenvalue of $M$, for some positive constant $C$ and exponents $p,q$, independent of $h$. Let $\tu^{(-1)}\in\R^d$ denote an estimate of the state at the initial time $t_0=0$. Suppose the initial parameter estimate $\Lam^{(0)}\in\PP(\Lam)$ satisfies
    \begin{align}\label{est:RLS_prior}
        \|\De\Lam^{(0)}\|\leq \rho_0,
    \end{align}
for some $\rho_0>0$. Then there there exists an increasing sequence of times $\{t_n\}_n$, along which \eqref{eq:model_error_estimate} holds, and positive constants $\{C_n\}_n$ such that 
    \begin{align}\label{eq:parameter_error_bound}   
         \|\Delta \Lam^{(n+1)}\| \leq D_*\frac{C_n\|(\tK^{(n)})^{-1}\|}{\mu_n^p}\|\De\Lam^{(n)}\|,
    \end{align}
for all $n\geq1$, for some positive constant $D_*$, provided that at each stage $\ell=0,\dots, n$, the nudging parameters $\mu_\ell$ are chosen such that
    \begin{align}\label{cond:Lam_RLS}
        \Lam^{(\ell+1)}\in\PP(\Lam),\quad \mu_n^p\geq C_0U_{n+1}\rho_0,
    \end{align}
Moreover, if $\mu_n$ is chosen such that
    \begin{align}\label{cond:delta_RLS}
        0< \de_n:=1-D_*\frac{C_n\|(\tK^{(n)})^{-1}\|}{\mu_n^p}<1,
    \end{align}
where $\sum_n\de_n=\infty$, then 
    \begin{align}\label{eq:converge_RLS}
        \lim_{n\rightarrow\infty}|\Delta\Lam^{(n)}|=0.
    \end{align}
\end{theorem}
    
\begin{proof} 
\add{
Let $\tu^0(t)=\tu(t;\tu^{(-1)},\Lam^{(0)},M^{(0)})$ denote the solution of \eqref{eq:nudge_system} corresponding to $\tLam=\Lam^{(0)}$, initial value $\tu^{(-1)}$, and matrix $M^{(0)}$.} 

\add{Let $t_1$ be large enough that \eqref{eq:model_error_estimate} holds for $t'=t_1$, $\tu_0=u^{(-1)}$, $\tLam=\Lam^{(0)}$, and such that $u(t_1;u_0)$ belongs to the absorbing ball of \eqref{eq:reference_system}, so that $U_1=\sup_{t\geq t_1}\|u(t;u_0)\|<\infty$. Let $\tu^{(0)}=\tu(t_1;\tu^{(-1)},\Lam^{(0)})$.} Then 
    \begin{align}\label{est:state_error_RLS_0}
        \|w^{(0)}\| \leq \frac{C_0}{\mu_0^{p}} \|\Delta \Lam^{(0)}\|,
    \end{align}
\add{where $\mu_0$ denotes the smallest eigenvalue of $M^{(0)}$}. We choose $\mu_0$ such that \eqref{eq:mu_choice_RNI_0} holds. Then \eqref{eq:w_RNI_0} also holds.

\add{From \eqref{eq:model_error_RLS}, we have}
    \begin{align}\notag
        \tilde{\mathcal{K}}^{(0)} \Delta \Lam^{(1)} = \tW^{(0)}\Lam + \tg^{(0)},
    \end{align}
First observe that
    \begin{align}\notag
        \tW_{jk}^{(0)} = \left\langle I_hL_kw^{(0)},I_hL_jw^{(0)}\right\rangle    +\left\langle I_hL_kw^{(0)},I_hL_ju^{(0)}\right\rangle    
    \end{align}
\add{Since $I_h$ and $L_j$ are bounded operators, it follows from \eqref{def:tK_tW} that}
    \begin{align}\label{eq:tW_RLS_0}
        \|\tW^{(0)}\| \leq \|I_h\|^2\|L\|^2\left(\|w^{(0)}\|+\|u^{(0)}\|\right)\|w^{(0)}\|\leq 2U_1\|I_h\|^2\|L\|^2\|w^{(0)}\|,
    \end{align}
\add{where $w^{(0)}=\tu^{(0)}-u^{(0)}$}. Now observe that
    \begin{align}\label{eq:decomposition_tW_RLS_0}
        \add{\tW^{(0)} \Lam=\tW^{(0)}\De\Lam^{(0)}+\tW^{(0)}\Lam^{(0)}}.
    \end{align}
Upon combining \eqref{eq:tW_RLS_0}, \eqref{eq:decomposition_tW_RLS_0} with \eqref{est:state_error_RLS_0}, \eqref{est:RLS_prior}, we obtain
    \begin{align}\label{est:tW_Lam_RLS_0}
        \|\tW^{(0)}\Lam\|&\leq\|\tW^{(0)}\|\left(\rho_0+\|\Lam^{(0)}\|\right)
        \notag
        \\
        &\leq 2U_1\|I_h\|^2\|L\|^2\left(\rho_0+\|\Lam^{(0)}\|\right)\|w^{(0)}\|\notag
        \\
        &\leq \frac{2C_0U_1}{\mu_0^p}\|I_h\|^2\|L\|^2\left(\rho_0+\|\Lam^{(0)}\|\right)\|\De\Lam^{(0)}\|.
    \end{align}
    
Next, let $C_F^{1}$ denote the local Lipschitz constant of $F$ {corresponding to the ball of radius $2U_1$}. Then from \eqref{eq:tW_RLS_0}, we have
    \begin{align}\label{est:tg_RLS_0}
        \|\tg^{(0)}\| 
        &\leq \|I_h\|^2\|L\|\|F(w^{(0)}+u^{(0)}) - F(u^{(0)})\|
        \notag
        \\
        &\leq C^1_F\|I_h\|^2\|L\|\|w^{(0)}\|\notag
        \\
        &\leq \frac{C_0C^1_F}{\mu_0^p}\|I_h\|^2\|L\|\|\De\Lam^{(0)}\|.
    \end{align}
Therefore, upon combining \eqref{est:tW_Lam_RLS_0}, \eqref{est:tg_RLS_0}, we arrive at
    \begin{align}\notag
        \|\Delta \Lam^{(1)}\|
        &\leq \frac{C_0\|(\tK^{(0)})^{-1}\|}{\mu_0^q}\|I_h\|\|L\|\left\{\left[2U_1\|L\|\left(\rho_0+\|\Lam^{(0)}\|\right)+C^1_F\right]\|I_h\|\right\}\|\De\Lam^{(0)}\|
    \end{align}
We then let
    \[
        D_* := \|I_h\|\|L\|\left\{\left[2U_1\|L\|\left(\rho_0+\|\Lam^{(0)}\|\right)+C^1_F\right]\|I\|\right\}, 
    \]
so that we may write
    \[
        \add{\|\Delta \Lam^{(1)}\| \leq  D_*\frac{C_0\|(\tK^{(0)})^{-1}\|}{\mu_0^{p}} \|\Delta \Lam^{(0)}\|.}
    \]  
\add{Hence, we choose $\mu_0$ } such that
    \begin{align}\label{cond:delta_RLS_0}
        1-\de_0=D_*\frac{C_0\|(\tK^{(0)})^{-1}\|}{\mu_0^{p}}<1,
    \end{align}
and, simultaneously, such that $\Lam^{(1)}\in\PP(\Lam)$. We proceed inductively.

Suppose that \eqref{eq:parameter_error_bound} and \eqref{cond:delta_RLS_0} hold for all $m = 0, \ldots, n-1$. Consider the case $m = n$. \add{Let $\tu^{n}(t)=\tu(t;\tu^{(n-1)},\Lam^{(n)},M^{(n)})$ denote the solution of \eqref{eq:nudge_system} corresponding to $\tLam=\Lam^{(n)}$, initial value $\tu^{(n-1)}$, and matrix $M^{(n)}$.} Let $t_{n+1} > 0$ \add{be} large enough \add{so} that \eqref{eq:model_error_estimate} holds \add{for $t'=t_{n+1}$, $\tu_0=u^{(n-1)}$, $\tLam=\Lam^{(n)}$}. Then 
    \[
        \|w^{(n)}\| \leq \frac{C_n}{\mu_n^p} \|\Delta \Lam^{(n)}\|,
    \]
\add{where $\mu_n$ denotes the smallest eigenvalue of $M_n$}. Under the induction hypotheses, we may argue as in \eqref{est:state_error_RNI_n} to deduce that
    \begin{align}\notag
        \|w^{(n)}\|\leq \frac{C_n\rho_0}{\mu_n^p}.    
    \end{align}
We then also assume that \eqref{eq:mu_choice_RNI_n} holds. Recall \eqref{eq:model_error_RLS}, which asserts that
    \[
        \tK^{(n)}\Delta \Lam^{(n+1)} = \tW^{(n)} \Lam +\tg^{(n)}.
    \]
\add{With this and \eqref{eq:mu_choice_RNI_n} in hand, we may now argue as in the base case to obtain}
    \begin{align}
        \|\tW^{(n+1)}\Lam\|&\leq\frac{2C_nU_1}{\mu_n^p}\|I_h\|^2\|L\|^2\left(\rho_0+\|\Lam^{(0)}\|\right)\|\De\Lam^{(n)}\|\notag
        \\
        \|\tg^{(n)}\|&\leq\frac{C_nC^1_F}{\mu_n^p}\|I_h\|^2\|L\|\|\De\Lam^{(n)}\|,\notag
    \end{align}
\add{where we implicitly used the fact that $U_{j}\leq U_1$, for all $j\geq0$, which implies $C_F^{j}\leq C_F^1$}. 
Thus 
    \begin{align}\notag
        \|\Delta \Lam^{(n+1)}\|
        &\leq D_*\frac{C_n\|(\tK^{(n)})^{-1}\|}{\mu_n^p}\|\De\Lam^{(n)}\|.
    \end{align}
We conclude the proof by taking $\mu_n$ such that 
    \begin{align}\notag
        1-\de_n=D_*\frac{C_n\|(\tK^{(n)})^{-1}\|}{\mu_n^p}<1,
    \end{align}
and, simultaneously, such that $\Lam^{(n+1)}\in\PP(\Lam)$. This completes the proof.
\end{proof}

\section{The two-layer Lorenz 96 (2L-L96) model}\label{sect:L96}

\subsection{L96 and its Dissipativity}
The original Lorenz-96 model is given by
\begin{equation}\label{eq:L96}
\frac{du_k}{dt} = u_{k-1}(u_{k+1}-u_{k-2}) - d_k u_k + F,
\end{equation}
where the $u_k$ are periodic with period $K$.  Each of the $u_k$ represents the discrete values of the zonal velocity of the atmosphere along a latitudinal circle where ${d}_j$ are the dissipation coefficients and $F$ is a constant large-scale forcing.  The model is more dominantly used as a canonical example of nonlinear dynamics, and has proven on many occasions to be a useful tool when investigating various forms of data assimilation and other data-driven algorithms (see \cite{grooms2015framework,Lguensat_Tandeo_Ailliot_Pulido_Fablet_2017,Brajard_Carrassi_Bocquet_Bertino_2020,Grooms_2021,stanley2021multivariate} for example).  This standard L96 model is not conducive to the AOT style of data assimilation however as it would require observation of nearly all the $u_j$ for full convergence of the other variables to occur.

\begin{figure}
    \centering
    \includegraphics[trim={0 75 20 80},clip, width=.45\textwidth]{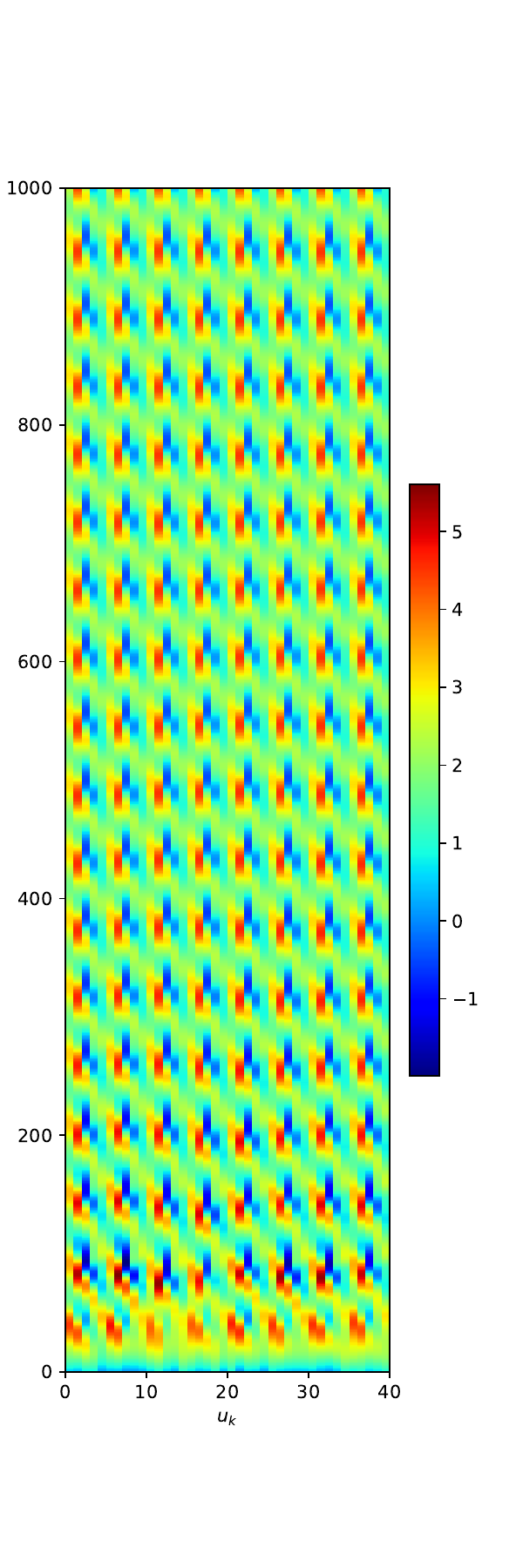}
    \caption{The dynamics of the two-layer L96 model visualized across the 40 slow scale $u_k$, plotted against time.  For this simulation, we used 5 $v_{k,j}$'s for each $k$.  All other parameters were selected as the default values from \cite{Chen_Majda_2018}.}
    \label{fig:L96_dynamics}
\end{figure}

We will instead focus on the two-layer Lorenz 96 model \cite{Wilks2005,Arnold_Moroz_Palmer_2013}. This model introduces a new set of small-scale variables $v_{k,j}$ that \add{characterize the evolution of small-scale waves, which, in turn, feed back nonlinearly into the evolution of the large-scale waves $u_k$}. We choose to study a particular variation of two-layer L96 considered in \cite{Chen_Majda_2018} for its conditional Gaussian structure, but whose analytical study appears to be underdeveloped. The particular system of \add{interest is given by:}
\begin{align}\label{eq:L96_slow}
\frac{du_k}{dt} &= u_{k-1}(u_{k+1}-u_{k-2}) + \sum_{j=1}^J \gam_{k,j}v_{k,j} u_k - d_ku_k + F,\\ \label{eq:L96_fast}
\frac{dv_{k,j}}{dt} &= -d_{k,j}v_{k,j} - \gamma_j u_k^2,
\end{align}
for $k=0,\dots K$ and $j=1,\dots J$, with the periodic boundary conditions $u_{-1}=u_K$ and $v_{-1,j} = v_{K,j}$ for all $j$, where the parameters satisfy:
    \begin{align}\label{eq:L96_assumptions}
        \gam_{k,j}=\gam_j,\quad\text{and}\quad \gam_j, d_k, d_{k,j}>0,\quad \text{for all}\ k=0,\dots, K, \quad j=1,\dots, J.
    \end{align}
Recall that the main goal is to use the RNI or RLS algorithm to reconstruct a given subset of unknown parameters $d_k$, $d_{k,j}$ \add{using} observations \add{on their} corresponding variables.

First let us show that \eqref{eq:L96_slow}--\eqref{eq:L96_fast} possesses an absorbing ball, which implies that the system is a dissipative dynamical system. For this, let 
    \begin{align}\label{def:d_star}
        d_*:=\min_k\left\{d_k,\ \min_jd_{k,j}\right\}.
    \end{align}
We will also make use of the notation
    \begin{align}\label{def:norm_u_v}
        \|u\|^2=\sum_{k=0}^Ku_k^2,\quad \|v\|^2=\sum_{k=0}^K\sum_{j=1}^Jv_{k,j}^2.
    \end{align}

\begin{lemma}\label{lem:L96_absorbing_ball}
For all $t\geq0$, it holds that
    \begin{align}\label{est:apriori:estimate}
        \|u(t)\|^2+\|v(t)\|^2\leq e^{-d_*t}\left(\|u(0)\|^2+\|v(0)\|^2\right)+K\frac{F^2}{d_*^2}(1-e^{-d_*t}).
    \end{align}
In particular, if $\|u(0)\|^2+\|v(0)\|^2\leq R^2$, then 
    \begin{align}\label{est:absorbing_ball}
        \sup_{t\geq T}\left(\|u(t)\|^2+\|v(t)\|^2\right)\leq 2K\frac{F^2}{d_*^2}=:\rho_*^2,
    \end{align}
for some $T(R)>0$, and if $R=\rho_*$, then \begin{align}\label{est:L96_bound}
        \sup_{t\geq0}\left(\|u(t)\|^2+\|v(t)\|^2\right)\leq \rho_*^2.
    \end{align}
\end{lemma}

\begin{proof}

Upon multiplying \eqref{eq:L96_slow} by $u_k$ and \eqref{eq:L96_fast} by $v_{k,j}$, summing over $k$ and $k,j$, respectively, then invoking periodicity, we obtain the energy balance
    \begin{align}\label{eq:L96_energy}
        &\frac{1}2\frac{d}{dt}\left(\|u\|^2+\|v\|^2\right)=-\sum_{k=0}^Kd_ku_k^2-\sum_{k=0}^K\sum_{j=1}^Jd_{k,j}v_{k,j}^2+F\sum_{k=0}^Ku_k.
    \end{align}
By repeated application of the Cauchy-Schwarz inequality, it thus follows that
    \begin{align}\notag
         &\frac{d}{dt}\left(\|u\|^2+\|v\|^2\right)+d_*\left(\|u\|^2+\|v\|^2\right)\leq K\frac{F^2}{d_*}.
    \end{align}
An application of Gronwall's inequality then yields the estimate
    \begin{align}\notag
        \|u(t)\|^2+\|v(t)\|^2\leq e^{-d_*t}\left(\|u(0)\|^2+\|v(0)\|^2\right)+K\frac{F^2}{d_*^2}(1-e^{-d_*t}),
    \end{align}
which is \eqref{est:apriori:estimate}. Finally, choosing $t$ sufficiently large, implies \eqref{est:absorbing_ball}.
\end{proof}

From \cref{lem:L96_absorbing_ball} and the L96 system, we also obtain bounds on the time-derivative of the state variables. For this, let
    \begin{align}\label{def:d_max}
        \|d\|^2:=\sum_{k=0}d_k^2+\sum_{k=0}^K\sum_{j=1}^Jd_{k,j}^2.
    \end{align}
To simplify the presentation we will also make use of the notation
    \begin{align}\notag
        \frac{d}{dt} f = \dot{f}
    \end{align}

\begin{corollary}
Suppose that $(u(0),v(0))$ belongs to the absorbing ball of \eqref{eq:L96_slow}--\eqref{eq:L96_fast}, whose radius is determined by \cref{lem:L96_absorbing_ball}. Then
    \begin{align}\label{est:L96_bound_dot}
        \sup_{t\geq 0}\left(\|\du(t)\|^2+\|\dv(t)\|^2\right)\leq 4\left[\left(1+\|\gam\|^2\right)\rho_*^2+\|d\|^2+\|\gam\|^2\right]\rho_*^2:=\dot{\rho}_*^2
    \end{align}
\end{corollary}

\begin{proof}
From \eqref{eq:L96_slow}--\eqref{eq:L96_fast}, the Cauchy-Schwarz inequality implies
    \begin{align}
        |\du_k|^2&\leq 2\left(|u_{k-1}|^2|u_{k+1}|^2+|u_{k-1}|^2|u_{k-2}|^2+\|\gam\|^2|u_k|^2\sum_{j=1}^Jv_{k,j}^2+d_k^2|u_k|^2+F^2\right)\notag
        \\
        |\dv_{k,j}|^2&\leq 2\left(d_{k,j}^2|v_{k,j}|^2+\gam_j^2|u_k|^2\right).\notag
    \end{align}
Upon summing over $k$ and $(k,j)$, then invoking periodicity and \eqref{est:L96_bound}, we obtain
    \begin{align}
        \|\du\|^2&\leq 2\left(2\|u\|^4+\|\gam\|^2\|u\|^2\|v\|^2+\|d\|^2\|u\|^2+KF^2\right)\notag
        \\
        &\leq 4\left(\left(1+\|\gam\|^2\right)\rho_*^2+\|d\|^2\right)\rho_*^2\notag
        \\
        \|\dv\|^2&\leq 2\left(\|d\|^2\|v\|^2+\|\gam\|^2\|u\|^2\right)\notag
        \\
        &\leq 2\left(\|d\|^2+\|\gam\|^2\right)\rho_*^2\notag.
    \end{align}
Combining these estimates yields \eqref{est:L96_bound_dot}
\end{proof}

\subsection{The Nudged 2L-L96 System}

Returning to the problem of parameter reconstruction in \eqref{eq:L96_slow}--\eqref{eq:L96_fast}, \add{we shall assume that} all large-scale variables are observed, but only a subset of small-scale variables are observed. The scope of the RNI and RLS algorithms is such that they will be able to reconstruct all large-scale damping parameters $d_k$, but only the small-scale damping parameters $d_{k,i}$ whose corresponding small-scale variable is directly observed.

Now the \add{nudging} system corresponding to \eqref{eq:L96_slow}--\eqref{eq:L96_fast} on which the two parameter reconstruction algorithms are based \add{is given by}
    \begin{align}\label{eq:L96_slow_nudge}
    \frac{d\tu_k}{dt} &= \tu_{k-1}(\tu_{k+1}-\tu_{k-2}) + \sum_{j=1}^J \gam_{k,j}\tilde{v}_{k,j} \tu_k - \td_k\tu_k + F - \add{\mu_k}(\tu_k - u_k),
    \\ 
    \frac{d\tilde{v}_{k,j}}{dt} &= -\td_{k,j}\tilde{v}_{k,j} - \gamma_j \tu_k^2-\add{\mu_{k,j}\chi_{\II}(k,j)(\tv_{k,j}-v_{k,j})},\label{eq:L96_fast_nudge}
    \end{align}
where \add{$\II\subset\{0,\dots, K\}\times\{1,\dots, J\}$ denotes the subset of index pairs representing which of the fast variables are directly observed, and $\chi_{\II}$ denotes the characteristic function over $\II$. Note that $\II$ may be empty, in which case only the large-scale damping parameters, $d_{k}$, will be reconstructed. In general,   
    \begin{align}\label{eq:II_structure}
        \II=\bigcup_{k=0}^K\JJ_k,\quad \text{where}\  \JJ_k\subset\{(k,1),\dots, (k,J)\}.
    \end{align}
Under these assumptions, observations constitute the following fraction of the total number of state variables:
    \begin{align}\label{eq:percentage}
        \frac{(K+1)+|\II|}{(K+1)(J+1)}.
    \end{align}

\begin{remark}
In the case $J=2$ and no fast variables are observed, then $|\II|=0$ and we observe only $1/3$ of the total state of the system; this particular case is analogous to the situation considered in \cite{carlson2020parameter}, where the Lorenz 63 model \cite{lorenz1963deterministic} was the model of interest and it was assumed that the $x$-variable was observed.
\end{remark}

We will furthermore assume that the proxy values $\td_k, \td_{k,j}$ satisfy
    \begin{align}\label{def:tdk_tdkj}
        \td_k\neq d_k,\qquad
        \td_{k,j}=\begin{cases}\td_{k,j}\neq d_{k,j},&(k,j)\in\II\\
        d_{k,j},&(k,j)\notin\II.
        \end{cases}
    \end{align}
Observe that under \eqref{eq:L96_assumptions}, the set of admissible parameters, $\PP(\Lam)$, for \eqref{eq:L96_slow}--\eqref{eq:L96_fast}, where $\Lam=(\{d_k\}_{k=0}^K,\{d_{k,j}\}_{(k,j)\in\II})$ represents the reference parameters, is an open set. Lastly, in the context of the general system \eqref{eq:reference_system}, the state variables of \eqref{eq:L96_slow}--\eqref{eq:L96_fast} are arranged as 
    \begin{align}\label{def:u_v}
        (u,v):=(u_0,\dots, u_K, v_{0,1},\dots v_{0,J},\dots, v_{K,1},\dots, v_{K,J}).
    \end{align}
Thus, the matrices $L_k$ are simply the elementary diagonal matrices of size $(K+1)(J+1)\times (K+1)(J+1)$ consisting of a single non-zero entry equal to $1$ in the position corresponding to the unknown parameter. Lastly, the observation operator $I_h$ is simply the projection operator defined by
    \begin{align}\label{def:Ih}
        I_h(u,v)=(u,(\chi_{\II}(k,j)v_{k,j})_{k,j}).
    \end{align}
}

\subsection{Convergence Analysis of RNI and RLS}

In light of \cref{thm:RNI} and \cref{thm:RLS}, in order to guarantee  convergence, it suffices to verify that the estimates \eqref{eq:model_error_estimate_RNI}, \eqref{eq:model_error_estimate} both hold. This is our current task. This will be established by \cref{lem:state_error_a}, \cref{lem:state_error_dot} below.

Denote the state error variables by 
    \begin{align}\label{def:state_error}
        w_k = \tu_k - u_k,\quad z_{k,j}=\tv_{k,j}-v_{k,j},
    \end{align}
and the model errors by
    \begin{align}\label{def:model_error}
        \De d_k=\td_k-d_k, \quad \De d_{k,j}=\td_{k,j}-d_{k,j}.
    \end{align}
We also make use of the notation
    \begin{align}\label{def:gam}
        \|\gam\|^2:=\sum_{j=0}^J\gam_j^2=\frac{1}{K+1}\sum_{k=0}^K\sum_{j=1}^J\gam_{k,j}^2,
    \end{align}
and
    \begin{align}\label{def:model_error_norm}
        \|\De d\|^2:=\sum_{k=0}^K(\De d_k)^2+\sum_{k=0}^K\sum_{j\in\JJ_k}(\De d_{k,j})^2.
    \end{align}
\add{Then the evolution of the state error is governed by}
\begin{align}
\frac{dw_k}{dt} &= w_{k-1}(w_{k+1}-w_{k-2})+w_{k-1}(u_{k+1}-u_{k-2}) +u_{k-1}(w_{k+1}-w_{k-2})\notag
    \\
    &\quad+\sum_{j=1}^J \gam_{k,j}(z_{k,j}w_k+z_{k,j}u_k +v_{k,j}w_k)- (\De d_k)\tu_k-d_kw_k  - \mu w_k \label{eq:L96_state_error_slow}
    \\ 
\frac{dz_{k,j}}{dt}&=-(\De d_{k,j})\tv_{k,j}-d_{k,j}z_{k,j}-\gam_jw_k^2-2\gam_ju_kw_k-\mu_{k,j}\chi_{\II}(k,j)z_{k,j}.\label{eq:L96_state_error_fast}
\end{align}

\begin{lemma}\label{lem:state_error_a}
Assume that $(u(0),v(0))$ belongs to the absorbing ball of \eqref{eq:L96_slow}--\eqref{eq:L96_fast}, whose radius is determined by \cref{lem:L96_absorbing_ball}. Suppose also that there exists $\de>0$ such that
    \begin{align}\label{cond:prior}
        \|\De d\|\leq \de.
    \end{align}
If 
\begin{align}\label{cond:mu_state_error_a}
        \mu_k\geq \max\left\{\frac{\|\gam\|^2\rho_*^2}{d_{k,j}}, 4M+2(2+(K+1)\|\gam\|)\rho_*\right\},\quad \min_{(k,j)\in\II}\mu_{k,j}\geq 2\de,
    \end{align}
for all $k=0,\dots, K$, then
    \begin{align}\label{est:L96_state_error_a}
        \|w(t)\|^2+\|z(t)\|^2\leq e^{-d_*t}\left(\|w(0)\|^2+\|z(0)\|^2\right)+\frac{\rho_*^2}{d_*\mu_*}\|\De d\|^2(1-e^{-d_*t}),
    \end{align}
where $\mu_*=\min_k\{\mu_k,\ \min_j\mu_{k,j}\}$. In particular, if $\|(w(0),z(0))\|\leq R$, then 
    \begin{align}\label{est:L96_state_error_b}
        \sup_{t\geq T}\left(\|w(t)\|^2+\|z(t)\|^2\right)\leq 2\frac{\rho_*^2}{d_*\mu_*}\|\De d\|^2=:\frac{\eta_*^2}{\mu_*}\|\De d\|^2
    \end{align}
for some $T(R)>0$.
\end{lemma}

\begin{proof}

Recall \eqref{eq:II_structure}, the periodic boundary conditions, and \eqref{eq:L96_assumptions}. Then \add{the state error satisfies the following energy balance:
    \begin{align}
        \frac{1}2&\frac{d}{dt}\left(\|w\|^2+\|z\|^2\right)+\sum_{k=0}^Kd_kw_k^2+\sum_{k=0}^K\sum_{j=1}^Jd_{k,j}z_{k,j}^2+\sum_{k=0}^K\mu_kw_k^2+\sum_{k=0}^K\sum_{j\in\JJ_k}\mu_{k,j}z_{k,j}^2\notag
        \\
        &=\sum_{k=0}^K\left(w_{k-1}w_ku_{k+1}-w_{k-2}u_{k-1}w_k\right)+\sum_{k=0}^K\sum_{j=1}^J\gam_{k,j}v_{k,j}w_k^2-\sum_{k=0}^K\sum_{j=1}^J\gam_ju_kw_kz_{k,j}\notag
        \\
        &\quad -\sum_{k=0}^K(\De d_k)\tu_kw_k-\sum_{k=0}^K\sum_{j\in\JJ_k}(\De\td_{k,j})\tv_{k,j}z_{k,j}\notag
    \end{align}
}
By repeated application of the Cauchy-Schwarz inequality and \eqref{est:L96_bound}, we estimate
    \begin{align}
        \sum_{k=0}^K\left(w_{k-1}w_ku_{k+1}-w_{k-2}u_{k-1}w_k\right)&\leq 2\rho_*\|w\|^2\notag
        \\
        \sum_{k=0}^K\sum_{j=1}^J\gam_{k,j}v_{k,j}w_k^2&\leq (K+1)\|\gam\|\rho_*\|w\|^2\notag
        \\
        -\sum_{k=0}^K\sum_{j=1}^J\gam_ju_kw_kz_{k,j}&\leq \frac{1}4\sum_{k=0}^K\mu_kw_k^2+\sum_{k=0}^K\frac{u_k^2}{\mu_k}\left(\sum_{j=1}^J\gam_jz_{k,j}\right)^2\notag\\
        &\leq \frac{1}4\sum_{k=0}^K\mu_kw_k^2+\|\gam\|^2\rho_*^2\sum_{k=0}^K\frac{1}{\mu_k}\|z_k\|^2\notag.
    \end{align} 
For the remaining terms involving the model error, first observe that
    \begin{align}
            -\sum_{k=0}^K(\De d_k)\tu_kw_k&=-\sum_{k=0}^K(\De d_k)w_k^2-\sum_{k=0}^K(\De d_k)u_kw_k\notag
            \\
            -\sum_{k=0}^K\sum_{j\in\JJ_k}(\De d_{k,j})\tv_{k,j}z_{k,j}&= -\sum_{k=0}^K\sum_{j\in\JJ_k}(\De d_{k,j})z_{k,j}^2 -\sum_{k=0}^K\sum_{j\in\JJ_k}(\De d_{k,j})v_{k,j}z_{k,j}.\notag
    \end{align}
We may then estimate
    \begin{align}
        -\sum_{k=0}^K(\De d_k)u_kw_k&\leq \frac{\rho_*^2}2\sum_{k=0}^K\frac{(\De d_k)^2}{\mu_k}+\frac{1}2\sum_{k=0}^K\mu_kw_k^2\notag
        \\
         -\sum_{k=0}^K\sum_{j\in\JJ_k}(\De\td_{k,j})v_{k,j}z_{k,j}&\leq \frac{\rho_*^2}2\sum_{k=0}^K\sum_{j\in\JJ_k}\frac{(\De d_{k,j})^2}{\mu_{k,j}}+\frac{1}2\sum_{k=0}^K\sum_{j\in\JJ_k}\mu_{k,j}z_{k,j}^2.\notag
    \end{align}
Upon combining the above estimates, we arrive at
    \begin{align}
        &\frac{1}2\frac{d}{dt}\left(\|w\|^2+\|z\|^2\right)+\sum_{k=0}^K\sum_{j=1}^J\left(d_{k,j}-\frac{\|\gam\|^2\rho_*^2}{\mu_k}\right)z_{k,j}^2+\sum_{k=0}^K\sum_{j\in\JJ_k}\left(\frac{\mu_{k,j}}{2}+(\De d_{k,j})\right)z_{k,j}^2\notag
        \\
        &\leq \frac{\rho_*^2}2\left[\sum_{k=0}^K\frac{(\De d_k)^2}{\mu_k}+\sum_{k=0}^K\sum_{j\in\JJ_k}\frac{(\De d_{k,j})^2}{\mu_{k,j}}\right]-\sum_{k=0}^K\left(d_k+\frac{\mu_k}4+(\De d_k)-\left(2+K\|\gam\|\right)\rho_*\right)w_k^2.\notag
    \end{align}
One now sees that \eqref{cond:mu_state_error_a} and an application of Gronwall's inequality  implies \eqref{est:L96_state_error_a}.
\end{proof}

Next, we study the time-derivative of the state error. For this, it will be useful to recall the following notation:
    \begin{align}\label{def:dot_w_v}
        \dw=\frac{d}{dt}w,\quad \dz=\frac{d}{dt}z.
    \end{align}
It will also be useful to separate the evolution of fast variables into its observed component and unobserved component
Then the time-derivative of the state errors is governed by the following system:
    \begin{align}
        \frac{d\dw_k}{dt} &= \dw_{k-1}(w_{k+1}-w_{k-2})+w_{k-1}(\dw_{k+1}-\dw_{k-2})\notag
        \\
        &\quad+\dw_{k-1}(u_{k+1}-u_{k-2})+w_{k-1}(\du_{k+1}-\du_{k-2})\notag
        \\
        &\quad+\du_{k-1}(w_{k+1}-w_{k-2})+u_{k-1}(\dw_{k+1}-\dw_{k-2})\notag
        \\
        &\quad+\sum_{j=1}^J \gam_{k,j}(\dz_{k,j}w_k+\dz_{k,j}u_k +\dv_{k,j}w_k+z_{k,j}\dw_k+z_{k,j}\du_k +v_{k,j}\dw_k)\notag
        \\
        &\quad- (\De d_k)\dtu_k-d_k\dw_k  - \mu \dw_k \label{eq:L96_state_error_dot_slow}
        \\
        \frac{d\dz_{k,j}}{dt}&=-(\De d_{k,j})\dtv_{k,j}-d_{k,j}\dz_{k,j}-2\gam_jw_k\dw_k-2\gam_j\du_kw_k-2\gam_ju_k\dw_k-\mu_{k,j}\dz_{k,j},\quad (k,j)\in\II\label{eq:L96_state_error_dot_fast_obs}
        \\
        \frac{d\dz_{k,j}}{dt}&=-(\De d_{k,j})\dtv_{k,j}-d_{k,j}\dz_{k,j}-2\gam_jw_k\dw_k-2\gam_j\du_kw_k-2\gam_ju_k\dw_k,\quad(k,j)\notin\II.\label{eq:L96_state_error_dot_fast_unobs}
    \end{align}
We observe that from \cref{lem:L96_absorbing_ball} and directly from the systems \eqref{eq:L96_slow}--\eqref{eq:L96_fast} and \eqref{eq:L96_state_error_slow}--\eqref{eq:L96_state_error_fast} that the time derivatives $(\dw,\dz)$ are well-defined and uniformly bounded over any finite-time interval. Our present goal is to obtain bounds on $(\dw,\dz)$ which depend in a favorable way on the nudging parameters, after possibly passing a transient time. 

\begin{lemma}\label{lem:state_error_dot}
Suppose that each of the hypotheses in \cref{lem:L96_absorbing_ball} are satisfied, so that both \eqref{est:L96_bound} and \eqref{est:L96_state_error_b} hold. Further suppose that
    \begin{align}\label{cond:mu_dot_a}
        \mu_*\geq 2\de
    \end{align}
and that
    \begin{align}\label{cond:mu_dot_b}
        \mu_*^3\geq\frac{16(4+\|\gam\|)\rho_*}3\mu_*^2+\frac{48\eta_*\de}{3}\mu_*^{3/2}+64\|\gam\|^2\left(\frac{\eta_*^2\de^2}{3d_*}+\rho_*^2\right)\mu_*+8\|\gam\|^2\eta_*^2\left(\frac{32\de^2}{3}+1\right).
    \end{align}
Then
    \begin{align}\label{est:L96_state_error_dot_a}
        \|\dw(t)\|^2+\|\dz(t)\|^2\leq e^{-d_*(t-T)}\left(\|\dw(T)\|^2+\|\dz(T)\|^2\right)+\frac{\deta_*^2}{2\mu_*}\|\De d\|^2,
    \end{align}
for all $t\geq T$, where $T$ is the transient time obtained from \cref{lem:state_error_a}, and
    \begin{align}\label{def:eta_dot}
        \deta_*^2:=\frac{4\drho_*^2}{d_*}\left(\frac{2\eta_*^2\left(16+5\|\gam\|^2\right)}{\mu_*}+\frac{4\|\gam\|^2\eta_*^2}{d_*}+4\drho_*^2+1\right).
    \end{align}
In particular, there exists $T'\geq T$ such that
    \begin{align}\label{est:L96_state_error_dot_b}
        \sup_{t\geq T'}\left(\|\dw(t)\|^2+\|\dz(t)\|^2\right)\leq \frac{\deta_*^2}{\mu_*}\|\De d\|^2.
    \end{align}
\end{lemma}

\begin{proof}
The energy balance for the time-derivative of the state error 
is given by
    \begin{align}
        \frac{1}2&\frac{d}{dt}
        \|w\|^2 +\sum_{k=0}^K(\mu_k+d_k+\De d_k)\dw_k^2\notag
        \\
        &= \sum_{k=0}^K\dw_{k-1}\dw_kw_{k+1}-\dw_{k-2}w_{k-1}\dw_k\notag
        \\
        &\quad+\sum_{k=0}^K\dw_{k-1}\dw_ku_{k+1}-u_{k-2}\dw_{k-1}\dw_k+w_{k-1}\dw_k\du_{k+1}-\du_{k-2}w_{k-1}\dw_k\notag
        \\
        &\quad+\sum_{k=0}^K\du_{k-1}\dw_kw_{k+1}-w_{k-2}\du_{k-1}\dw_k+u_{k-1}\dw_k\dw_{k+1}-\dw_{k-2}u_{k-1}\dw_k\notag
        \\
        &\quad+\sum_{k=0}^K\sum_{j=1}^J \gam_{k,j}(\dz_{k,j}w_k\dw_k+\dz_{k,j}u_k\dw_k +\dv_{k,j}w_k\dw_k+z_{k,j}\dw_k^2+z_{k,j}\du_k\dw_k +v_{k,j}\dw_k^2)\notag
        \\
        &\quad- \sum_{k=0}^K(\De d_k)\du_k\dw_k\label{eq:L96_state_error_dot_balance_slow}
        \\
        \frac{1}{2}&\frac{d}{dt}\left(\sum_{k=0}^K\sum_{j\in\JJ_k}\dz_{k,j}^2\right)+\sum_{k=0}^K\sum_{j\in\JJ_k}(\mu_{k,j}+d_{k,j}+\De d_{k,j})\dz_{k,j}^2\notag
        \\
        &=-\sum_{k=0}^K\sum_{j\in\JJ_k}(\De d_{k,j})\dv_{k,j}\dz_{k,j}+2\gam_jw_k\dw_k\dz_{k,j}+2\gam_j\du_kw_k\dz_{k,j}\label{eq:L96_state_error_dot_balance_fast_obs}
        \\
        \frac{1}{2}&\frac{d}{dt}\left(\sum_{k=0}^K\sum_{j\notin\JJ_k}\dz_{k,j}^2\right)+\sum_{k=0}^K\sum_{j\notin\JJ_k}d_{k,j}\dz_{k,j}^2\notag
        \\
        &=-2\sum_{k=0}^K\sum_{j\notin\JJ_k}\gam_jw_k\dw_k\dz_{k,j}+\gam_j\du_kw_k\dz_{k,j}\label{eq:L96_state_error_dot_balance_fast_unobs}
    \end{align}
    
Suppose that $t\geq T$. We treat the above sign-indefinite terms by repeatedly apply the Cauchy-Schwarz and Young's inequality, periodicity, and the bounds \eqref{est:L96_bound}, \eqref{est:L96_state_error_b}. We estimate the above in groups organized according to their total order in $w$, $z$, $u$, $v$. First we treat the slow variables:
    \begin{align}\notag
        \textrm{cubic in }w\leq 2\|w\|\|\dw\|^2\leq \frac{2\eta_*\|\De d\|}{\mu_*^{1/2}}\|\dw\|^2\leq \frac{2\eta_*\de}{\mu_*^{1/2}}\|\dw\|^2
    \end{align}
    \begin{align}
        \textrm{quadratic in $w$, linear $u$}&\leq 4\|\du\|\|w\|\|\dw\|++4\|u\|\|\dw\|^2\notag
        \\
        &\leq \frac{64\|\du\|^2}{\mu_*}\|w\|^2+\frac{\mu_*}{16}\|\dw\|^2+4\rho_*\|\dw\|^2\notag
        \\
        &\leq\frac{64\drho_*^2\eta_*^2}{\mu_*^2}\|\De d\|^2+\left(\frac{\mu_*}{16}+4\rho_*\right)\|\dw\|^2\notag
    \end{align}
    \begin{align}
        \textrm{quadratic in $w$, linear in $v$}&\leq \|\gam\|\left(\|\dv\|\|w\|\|\dw\|+\|v\|\|\dw\|^2\right)\notag
        \\
        &\leq \frac{4\|\gam\|^2\|\dv\|^2}{\mu_*}\|w\|^2+\frac{\mu_*}{16}\|\dw\|^2+\rho_*\|\gam\|\|\dw\|^2\notag
        \\
        &\leq \frac{4\|\gam\|^2\drho_*^2\eta_*^2}{\mu_*^2}\|\De d\|^2+\left(\frac{\mu_*}{16}+\rho_*\|\gam\|\right)\|\dw\|^2\notag
    \end{align}
    \begin{align}
        \textrm{quadratic in $w$, linear in $z$}&\leq \|\gam\|\|\dz\|\|w\|\|\dw\|+\|z\|\|\dw\|^2\notag
        \\
        &\leq \frac{4\|\gam\|^2}{\mu_*}\|w\|^2\|\dz\|^2+\frac{\mu_*}{16}\|\dw\|^2+\frac{\eta_*\|\De d\|}{\mu_*^{1/2}}\|\dw\|^2\notag
        \\
        &\leq \frac{4\|\gam\|^2\eta_*^2}{\mu_*^2}\|\dz\|^2+\left(\frac{\mu_*}{16}+\frac{\eta_*\de}{\mu_*^{1/2}}\right)\|\dw\|^2\notag
    \end{align}
    \begin{align}
        \textrm{linear in $w$, $z$, and $v$, $u$}&\leq \|\gam\|\left(\|\dz\|\|u\|+\|z\|\|\du\|\right)\|\dw\|\notag
        \\
        &\leq \frac{8\|\gam\|^2\|u\|^2}{\mu_*}\|\dz\|^2+\frac{8\|\gam\|^2\|\du\|^2}{\mu_*}\|z\|^2+\frac{\mu_*}{16}\|\dw\|^2\notag
        \\
        &\leq \frac{8\|\gam\|^2\rho_*^2}{\mu_*}\|\dz\|^2+\frac{8\|\gam\|^2\drho_*^2\eta_*^2}{\mu_*^2}\|\De d\|^2+\frac{\mu_*}{16}\|\dw\|^2\notag
    \end{align}
    \begin{align}
        \textrm{linear in $w$, $u$}&\leq \|\De d\|\|\du\|\|\dw\|\leq \frac{4\drho_*^2}{\mu_*}\|\De d\|^2+\frac{\mu_*}{16}\|\dw\|^2\notag
    \end{align}
Combining these estimates in \eqref{eq:L96_state_error_dot_balance_slow} with \eqref{cond:mu_dot_a},\eqref{cond:mu_dot_b} yields
    \begin{align}\label{est:w_dot}
        &\frac{1}2\frac{d}{dt}\|\dw\|^2+d_*\|\dw\|^2+\left(\frac{3\mu_*}{16}-\frac{3\eta_*\de}{\mu_*^{1/2}}-(4+\|\gam\|)\rho_*\right)\|\dw\|^2\notag
        \\
        &\leq \left(\frac{4(16+3\|\gam\|^2)\drho_*^2\eta_*^2}{\mu_*^2}+\frac{4\drho_*^2}{\mu_*}\right)\|\De d\|^2+\left(\frac{4\|\gam\|^2\eta_*^2}{\mu_*^2}+\frac{8\|\gam\|^2\rho_*^2}{\mu_*}\right)\|\dz\|^2
    \end{align}

Next, we treat the observed fast variables:
    \begin{align}
        \textrm{quadratic in $w$, linear in $z$}&\leq 2\|\gam\|\|w\|\|\dw\|\|\dz\|\leq \frac{4\|\gam\|^2}{\mu_*}\|w\|^2\|\dw\|^2+\frac{\mu_*}{8}\sum_{k=0}^K\sum_{j\in\JJ_k}\dz_{k,j}^2\notag
        \\
        &\leq \frac{8\|\gam\|^2\eta_*^2\de^2}{\mu_*^2}\|\dw\|^2+\frac{\mu_*}{8}\sum_{k=0}^K\sum_{j\in\JJ_k}\dz_{k,j}^2\notag.
    \end{align}
    \begin{align}
        \textrm{linear in $w$, $z$, $u$}&\leq 2\|\gam\|\|\du\|\|w\|\|\dz\|\leq \frac{8\|\gam\|^2\|\du\|^2}{\mu_*}\|w\|^2+\frac{\mu_*}{8}\sum_{k=0}^K\sum_{j\in\JJ_k}\dz_{k,j}^2\notag
        \\
        &\leq \frac{8\|\gam\|^2\drho_*^2\eta_*^2}{\mu_*^2}\|\De d\|^2+\frac{\mu_*}8\sum_{k=0}^K\sum_{j\in\JJ_k}\dz_{k,j}^2.\notag
    \end{align}
    \begin{align}
        \textrm{linear $z$, linear in $v$}&\leq\|\De d\|\|\dv\|\|\dz\|\leq \frac{2\drho_*^2}{\mu_*}\|\De d\|^2+\frac{\mu_*}{8}\sum_{k=0}^K\sum_{j\in\JJ_k}\dz_{k,j}^2.\notag
    \end{align}
Combining these estimates in \eqref{eq:L96_state_error_dot_balance_fast_obs} with \eqref{cond:mu_dot_a},\eqref{cond:mu_dot_b} yields
    \begin{align}\label{est:z_dot_obs}
        \frac{1}{2}&\frac{d}{dt}\left(\sum_{k=0}^K\sum_{j\in\JJ_k}\dz_{k,j}^2\right)+d_*\sum_{k=0}^K\sum_{j\in\JJ_k}\dz_{k,j}^2+\frac{5\mu_*}8\sum_{k=0}^K\sum_{j\in\JJ_k}\dz_{k,j}^2\notag
        \\
        &\leq \frac{8\|\gam\|^2\eta_*^2\de^2}{\mu_*^2}\|\dw\|^2+2\drho_*^2\left(\frac{4\|\gam\|^2\eta_*^2}{\mu_*^2}+\frac{1}{\mu_*}\right)\|\De d\|^2.
    \end{align}

Lastly, we treat the unobserved fast variables. Notice that the same terms as in the observed fast variables appears, except the ones with prefactor $\De d_{k,j}$ and $\mu_{k,j}$. We thus estimate the remaining terms as
    \begin{align}
        \textrm{quadratic in $w$, linear in $z$}&\leq \frac{4\|\gam\|^2\eta_*^2\de^2}{\mu_*d_*}\|\dw\|^2+\frac{d_*}{4}\sum_{k=0}^K\sum_{j\notin\JJ_k}\dz_{k,j}^2,\notag
    \end{align}
    \begin{align}
        \textrm{linear in $w$, $z$, $u$}\leq \frac{4\|\gam\|^2\drho_*^2\eta_*^2}{\mu_*d_*}\|\De d\|^2+\frac{d_*}4\sum_{k=0}^K\sum_{j\notin\JJ_k}\dz_{k,j}^2.\notag
    \end{align}
Combining these estimates in \eqref{eq:L96_state_error_dot_balance_fast_unobs} yields
    \begin{align}\label{est:z_dot_unobs}
        \frac{1}{2}&\frac{d}{dt}\left(\sum_{k=0}^K\sum_{j\notin\JJ_k}\dz_{k,j}^2\right)+\frac{d_*}2\sum_{k=0}^K\sum_{j\notin\JJ_k}\dz_{k,j}^2\leq\frac{4\|\gam\|^2\eta_*^2\de^2}{\mu_*d_*}\|\dw\|^2+\frac{4\|\gam\|^2\drho_*^2\eta_*^2}{\mu_*d_*}\|\De d\|^2.
    \end{align}
    
Finally, we combine \eqref{est:w_dot}, \eqref{est:z_dot_obs}, \eqref{est:z_dot_unobs} to obtain
    \begin{align}
        &\frac{1}2\frac{d}{dt}\left(\|w\|^2+\|z\|^2\right)+\frac{d_*}2\left(\|w\|^2+\|z\|^2\right)\notag\\
        &\qquad +\left(\frac{3\mu_*}{16}-\frac{3\eta_*\de}{\mu_*^{1/2}}-(4+\|\gam\|)\rho_*-\frac{4\|\gam\|^2\eta_*^2\de^2}{\mu_*}\left(\frac{4}{\mu_*}+\frac{1}{d_*}\right)\right)\|\dw\|^2\notag
        \\
        &\qquad +\left(\frac{5\mu_*}8-\frac{4\|\gam\|^2\eta_*^2}{\mu_*^2}-\frac{8\|\gam\|^2\rho_*^2}{\mu_*}\right)\sum_{k=0}^K\sum_{j\in\JJ_k}z_{k,j}^2\notag
        \\
        &\leq \frac{2\drho_*^2}{\mu_*}\left(\frac{2\eta_*^2\left(16+5\|\gam\|^2\right)}{\mu_*}+\frac{4\|\gam\|^2\eta_*^2}{d_*}+4\drho_*^2+1\right)\|\De d\|^2.\notag
    \end{align}
A final application of \eqref{cond:mu_dot_a} and \eqref{cond:mu_dot_b} yields \eqref{est:L96_state_error_dot_a}. The estimate \eqref{est:L96_state_error_dot_b} follows immediately.
\end{proof}

\subsection{Numerical results for the two-layer L96 model}
The full code for this model is included in the supplementary material as a Python-based jupyter notebook which can be easily adapted as desired.  The notebook includes a basic simulator for the two-layer L96 model, a simulator for the data assimilated model with all parameters correct, and a data assimilated model where the parameters are updated according to the derivations provided above.  The differential equations are solved using a Runge-Kutta 45 scheme in Scipy's \verb|solve_ivp| function with the standard default settings.

Using the default parameters and dimensions specified in \cite{Chen_Majda_2018}, we see the development of highly nonlinear dynamics as indicated in Fig. \ref{fig:L96_dynamics}.  This is specified precisely by selecting $K=40$ slow variables $u_k$ with $J=5$ ($v_{k,j}$) fast variables for each of the $u_k$ resulting in a $240$ dimensional system.  The forcing is set to the universally constant value $F=5$, and the true dissipation and interaction coefficients are set to
\begin{align*}
    \tilde{d}_k &= 1.0 + 0.7\cos [2\pi(k+1)/J],\\
    d_{k,j} &= d_j,\quad d_1 = 0.2,~d_2 = 0.5,~d_3 = 1,~d_4=2,~d_5=5,\\
    \gamma_{k,j} &= \gamma_j = 0.1+0.25\cos(2\pi(k+1)/J).
\end{align*}
In each simulation, with and without nudging and/or parameter estimation, the initial conditions for both systems are drawn from the standard uniform distribution.

When the model is completely accurate (all parameters are known exactly) then the nudged system is able to identify all unobserved portions of the state (so long as the observed part of the state is in the slow $u_k$ variables) as indicated in Fig. \ref{fig:L96_AOT_error}.  The fast and slow scales interact noticeably as demonstrated in the supplemental notebook.  Stronger nonlinear effects can be noticed by increasing the strength of the forcing $F$, decreasing the dissipation coefficients, or varying any of the other parameters described above.  We do not report on any of these additional simulations here as they do not modify the primary conclusions of this article, but note that the reader is welcome to modify the supplemental notebook and investigate these features for themselves.

\begin{figure}
    \centering
    \includegraphics{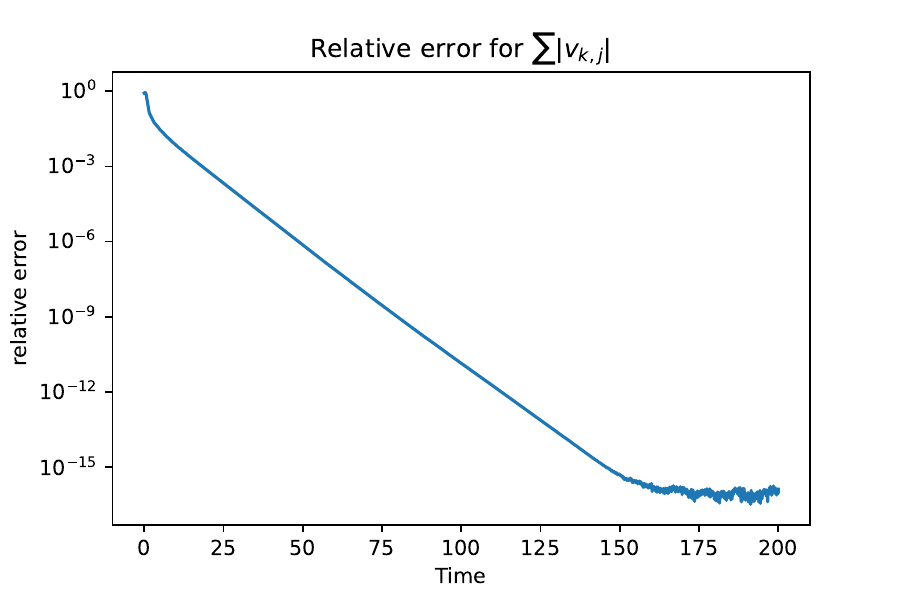}
    \caption{A plot of the relative error for the fast scale dynamics.  This is the relative norm for 20 of fast variables which indicates that the convergence of the nudged system to the true state is very near machine epsilon.}
    \label{fig:L96_AOT_error}
\end{figure}

To investigate the ability of the presented algorithm to identify unknown parameters, we consider the case where any subset of the $\td_k$ are unknown with the corresponding $u_k$ being observed.  In every case, so long as $u_k$ is observed, then the corresponding $\td_k$ can be approximated within a reasonable number of iterations of the algorithm.  This is irrespective of which and how many of the $u_k$ are observed or how many of the $\td_k$ are unknown.  An example is depicted in Fig. \ref{fig:L96_param_error} where the relative error in the parameters is plotted where the first 20 $\td_k$ are the unknown parameters.  Variations in the various hyper-parameters (length of time between parameter updates, value of the nudging coefficients $\mu$, and the initial guess for the unknown parameter) naturally lead to a different convergence rate, but the changes are not unexpected and can be explored in detail in the supplemental notebook.

\begin{figure}
    \centering
    \includegraphics{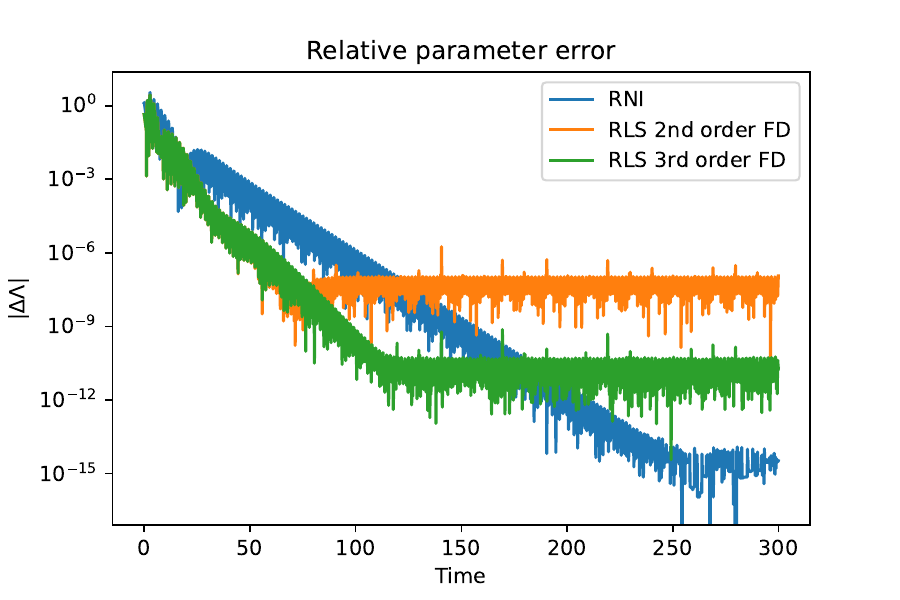}
    \caption{The relative parameter error $\|\Delta \Lam\|/\|\Lam\|$ as a function of time for the two-layer L96 model.  This is when the first 20 $\overline{d}_k$ are unknown, with both algorithms updating the parameter every $0.1$ time units with a fixed nudging coefficient of $\mu=50$.  Note that the error does not converge to machine epsilon for the RLS algorithm.  This is due to the choice of time step and is restricted by the need to numerically approximate the temporal derivative $I_h \frac{d}{dt}u$.  The final level of this error term can be adequately adjusted by choice of the enforced time step in the solver.  The final error from RNI is within an order of magnitude of machine epsilon, indicating adquate convergence to the exact value of the true parameters.}
    \label{fig:L96_param_error}
\end{figure}

As noted in \cite{pachev2022concurrent} the primary source of error in the RNI parameter update is the numerical approximation of the observable temporal derivative.  To identify this effect, we have included the option for three different backward finite differencing methods to approximate $\frac{du_k}{dt}$ which can be toggled between to gauge the relative error each produces.  These yield a corresponding first, second, and third order approximation.  Results from the 2nd and 3rd order approximation are shown in Fig. \ref{fig:L96_param_error}.  The order of the method doesn't seem to affect the convergence rate, but the final error value is directly due to this artifact. The same effect can be examined by forcing the selected solver (Scipy's built-in {solve\_ivp} in this case) to use a fixed time step, and then lowering the time step until the parameter error reaches machine epsilon.  Any and all of these options are open to the reader, and are straightforward to adjust in the provided supplemental notebook.

We note that the convergence rate for both RNI and RLS appears to be dictated by the choice of $\mu$, the smallest eigenvalue of the diagonal matrix $M$ (in these numerical simulations, $M = \mu I$ is a constant multiple of the identity matrix).  Although RNI does achieve an overall small error, its rate of convergence appears to be slightly worse than RLS (for either choice of the finite difference scheme). We also note that although the rigorous results stated above rely on the observation of all slow variables $u_k$, this is not necessary numerically, but only those slow variables whose corresponding dissipation coefficients $\tilde{d}_k$ are unknown, need to be observed.  This highlights the difference between the pessimistic estimates provided by the analysis, and the practical implementation of the algorithms described here.

\section{Application to the convection problem}\label{sect:RBC}
To investigate the utility of and compare and contrast both of the parameter recovery algorithms discussed above in the infinite-diemsional setting, we consider the paradigm of Rayleigh-B\'enard convection (RBC, see \cite{ahlers2009heat,chilla2012new}).  RBC is a classical example of chaotic (even turbulent) behavior in a physically simple setting, derived by considering the Boussinesq approximation of a fluid heated from below and cooled from above.  Lord Rayleigh first considered the mathematical setting of this problem by investigating the stability of the quiescent conductive state (absence of fluid motion) in \cite{rayleigh1916lix}.  Since that time, RBC has been a canonical problem not only for physicists to explore the nature of convective turbulence, but also as a testbed for pattern formation \cite{cross1993pattern}.  Of particular relevance to this current study, continuous data assimilation has been applied to this system in several different settings \cite{FarhatJollyTiti2015,FarhatLunasinTiti2016b,FarhatJohnstonJollyTiti2018} including when the full parameters of the system are unknown \cite{FarhatGlattHoltzMartinezMcQuarrieWhitehead2020}.  We adapt both the RNI and RLS algorithms to RBC to recover the two non-dimensional parameters of interest for the sytem: the Rayleigh number $\Ra$ and the Prandtl number $\Pr$.

For ease of implementation, and to allow for a more thorough exploration of parameter space, we restrict our attention to two spatial dimensions, and hence are interested in the system of equations given by
\begin{align*}
    \frac{1}{\Pr}\left(\frac{\partial \bv{u}}{\partial t} + \bv{u}\cdot \nabla \bv{u} + \nabla p\right) &= \Delta \bv{u} + \Ra \bv{e}_2 \theta, \quad \nabla \cdot \bv{u} = 0,\\
    \frac{\partial \theta}{\partial t} + \bv{u}\cdot \nabla \theta &= \Delta \theta.
\end{align*}
In this setting the velocity field $\bv{u} = (u,v)^T$ is two dimensional, the Prandtl number $\Pr$ is the ratio of the kinematic viscosity to the thermal diffusivity, and the Rayleigh number is a non-dimensional measure of the strength of the thermal forcing.  In all that follows, since we are considering the two-dimensional system, we will rewrite the evolution equations in terms of the vorticity, i.e.
\begin{align} \label{eq:vorticity_RB_system} \begin{split}
\frac{\p \zeta}{\p t} + \bv{u} \cdot \del \zeta 
&= \Pr \Delta \zeta + \Pr \Ra \theta_x, \\
\frac{\p \theta}{\p t} + \bv{u} \cdot \del \theta &= \Delta \theta,
\end{split} \end{align}
where $\zeta = \nabla \times \bf{u}$ is the vorticity of the system. For a more extended discussion on this system see \cite{murri2022}.

\subsection{Derivation of the RNI Algorithm for Estimating Multiple Parameters}\label{sec:CHL_multiparameter_derivation}

The RNI algorithm can be used for estimating multiple parameters but only if the system contains multiple equations, and each parameter of interest is found in exactly one equation. The derivation proceeds in a similar manner; each equation is used to derive a separate update. Note that in this case it is necessary to use a different nondimensionalization of the RBC system which allows $\Pr$ and $\Ra$ to be located in separate equations. In order to identify both parameters, we also necessarily require a nudging term on both the temperature and velocity equations. The system in this setting is given as follows: 
\begin{align} \label{eq:CHL_Pr_Ra_system} \begin{split}
\frac{\p \zeta}{\p t} + \bv{u} \cdot \del \zeta 
&= \Delta \zeta + \Ra \theta_x, \\
\frac{\p \theta}{\p t} + \bv{u} \cdot \del \theta &= \text{Pr}^{-1} \Delta \theta, \\
\frac{\p \tilde{\zeta}}{\p t} + \bv{\tu} \cdot \del \tilde{\zeta}
&= \Delta \tilde{\zeta} + \widetilde{\Ra} \, \tilde{\theta}_x - \mu_1 I_h\left(\tilde{\zeta} - \zeta\right), \\
\frac{\p \tilde{\theta}}{\p t} + \bv{\tu} \cdot \del \tilde{\theta} &= \widetilde{\Pr}^{-1} \Delta \tilde{\theta} - \mu_2 I_h \left( \tilde{\theta} - \theta \right)
\end{split} .
\end{align}
Following the general derivation of RNI as outlined above, we define $\bv{w} = \bv{\tu} - \bv{u}$, $z = \tilde{\zeta}- \zeta$, and $\eta = \tilde{\theta} - \theta$ and note that after waiting a sufficiently long time $t=t_1$, then the parameter update is given by: 
\[
\widetilde{\Ra}^{n+1} = \widetilde{\Ra}^n - \mu_1 \frac{\|I_h z\|^2}{\innerp{I_h \theta_x}{I_h z}} \Bigg \vert_{t=t_1}
\]
and
\[
\left(\widetilde{\Pr}^{-1}\right)^{n+1} = \left(\widetilde{\Pr}^{-1}\right)^n - \mu_2 \frac{\|I_h\eta\|^2}{\langle I_h\eta, I_h\Delta \theta\rangle}.
\]

However, a more refined RNI update can be provided if, rather than omitting the nonlinear terms entirely, we approximate them by the nonlinear effects on the observed state \cite{murri2022}, that is, we follow the approximation
\begin{equation}\notag
    I_h F(u) \approx I_h F(I_hu)
\end{equation}
in the general setting.  For the Rayleigh-B\'enard problem this leads to an update formula (which we refer to in this section as the $\mathrm{RNI^+}$ algorihm) of the form:
\begin{align*}
    \widetilde{\Ra}^{n+1} &= \widetilde{\Ra}^n + \frac{\left\langle I_h z, I_h(\tu) \cdot I_h(\nabla z) + I_hz\cdot I_h(\nabla \eta) - I_h \Delta w - \widetilde{\Ra}I_h \tilde{\eta}_x + \mu_1 I_h z\right\rangle}{\left\langle I_h \theta_x,I_h z\right\rangle}\\
    \left(\widetilde{\Pr}^{-1}\right)^{n+1} &= \left(\widetilde{\Pr}^{-1}\right)^n + \frac{\left\langle I_h\eta, I_h\tu\cdot I_h(\nabla \eta) + I_h w \cdot I_h (\nabla \theta) - \left(\widetilde{\Pr}^{-1}\right)^nI_h\Delta \eta + \mu_2 I_h \eta\right\rangle}{\left\langle I_h\eta,I_h \Delta \theta\right\rangle}.
\end{align*}

Interestingly, we report that the standard RNI algorithm did not achieve convergence down to machine precision in estimating both $\Pr$ and $\Ra$ in this setting, whereas the $\mathrm{RNI^+}$ algorithm described above did.  Since the RNI algorithm was derived under the assumption that the initial parameter guess is already sufficiently close to the true parameter (see discussion following \eqref{eq:sensitivity} in \cref{sect:RNI}), the observed failure of the RNI algorithm to converge, at least in the context of the RBC system, may be an indication that the algorithmic parameter regimes that admit convergence may be excessively constrictive, and therefore, difficult to enforce in practice. Furthermore, Theorem \cref{thm:RNI} requires the nudging matrix $M$ to have its smallest eigenvalue $\mu$ bounded below by a constant proportional to the initial parameter error. In the context of a numerical simulation with a certain computational budget, the magnitude of $\mu$ is, in practice, capped by the CFL condition on the time step.  

From this point of view, our numerical results for the $\mathrm{RNI^+}$ algorithm (see \cref{sect:simulations} below) may provide further relaxation to these conditions that are more conducive to its practical implementation. On the other hand, the failure of convergence may be an indication that not all terms that were dropped in the derivation are actually justifiable in this context. Indeed, the observed failure of the RNI algorithm to converge in the context of the RBC system is not in contradiction with \cref{thm:RNI}, which is only valid, in general, for finite-dimensional dissipative systems.

Regardless, this phenomenon deserves further investigation, particularly the extent to which the RNI framework can be generalized to the PDE setting. Indeed, in \cite{carlson2020parameter}, the standard RNI algorithm was applied to recover the kinematic viscosity in the 2D NSE, which suggests that the failure of RNI to converge may be subtle. This issue will be studied in a more systematic fashion in a follow-up work.

\subsection{Derivation of RLS for estimating both parameters in RBC}

Unlike RNI, the RLS algorithm is capable of estimating $\Pr$ and $\Ra$ simultaneously while nudging the vorticity (or velocity) alone. This seems to suggest that having direct access to information on the time-derivative of the data, which RLS makes explicit use of, holds significant advantages. 

Starting with this setup, i.e., not nudging the temperature directly, let $z = \tilde{\zeta}-\zeta$ be the error in vorticity, it follows that 
\[
\frac{\p z}{\p t} 
= \frac{\p \tilde{\zeta}}{\p t} - \frac{\p \zeta}{\p t}
= - \bv{\tu} \cdot \del \tilde{\zeta}
+ \widetilde{\Pr} \Delta \tilde{\zeta}
+ \widetilde{\Pr} \widetilde{\Ra} \tilde{\theta}_x 
- \mu I_h\left( \tilde{\zeta} - \zeta \right) - \zeta_t,
\]
where $\zeta_t := \p \zeta/\p t$. For the numercial simulations performed here, $I_h$ will denote Galerkin truncation to up wavenumbers of size $N=1/h$. Subsequently, $I_h$ is linear, commutes with derivatives, and is idempotent as $I_h$. Upon simplifying, then applying the observation operator $I_h$ throughout, it follows that 
\[
I_h z_t 
= \widetilde{\Pr} I_h(\Delta \tilde{\zeta})
+ \widetilde{\Pr} \widetilde{\Ra} I_h(\tilde{\theta}_x) -  \left[ I_h(\bv{\tu} \cdot \del \tilde{\zeta})
 + I_h \zeta_t \right] - \mu I_h z.
\]
At this point, the goal is to choose $\widetilde{\Pr}$ and $\widetilde{\Ra}$ to enforce 
\begin{align}\label{eq:RLS_RBC_condition}
\widetilde{\Pr} I_h(\Delta \tilde{\zeta})
+ \widetilde{\Pr} \widetilde{\Ra} I_h(\tilde{\theta}_x) -  \left[ I_h(\bv{\tu} \cdot \del \tilde{\zeta})
 + I_h\zeta_t \right] = 0.
\end{align}
to the extent possible. Following the previous discussion on the RNI algorithm, we settle on the following conditions which $\widetilde{\Pr}$ and $\widetilde{\Ra}$ must be chosen to satisfy:
\begin{align*}
\widetilde{\Pr} \|I_h(\Delta \tilde{\zeta})\|^2
+ \widetilde{\Pr} \widetilde{\Ra} \innerp{I_h(\tilde{\theta}_x)}{I_h(\Delta \tilde{\zeta})} &= \innerp{I_h(\bv{\tu} \cdot \del \tilde{\zeta})
 + I_h\zeta_t}{I_h(\Delta \tilde{\zeta})} \\
\widetilde{\Pr} \innerp{I_h(\Delta \tilde{\zeta})}{I_h(\tilde{\theta}_x)}
+ \widetilde{\Pr} \widetilde{\Ra} \|I_h(\tilde{\theta}_x)\|^2 &= \innerp{I_h(\bv{\tu} \cdot \del \tilde{\zeta})
 + I_h\zeta_t}{I_h(\tilde{\theta}_x)}
\end{align*}
which can be written in matrix form:
\begin{align}\label{eq:RLS_matrix_form}
\begin{pmatrix}
\|I_h(\Delta \tilde{\zeta})\|^2 & \innerp{I_h(\tilde{\theta}_x)}{I_h(\Delta \tilde{\zeta})} \\
\innerp{I_h(\Delta \tilde{\zeta})}{I_h(\tilde{\theta}_x)} & 
\|I_h(\tilde{\theta}_x)\|^2
\end{pmatrix}
\begin{pmatrix}
\widetilde{\Pr} \\ \widetilde{\Pr}\widetilde{\Ra}
\end{pmatrix}
=
\begin{pmatrix}
\innerp{I_h(\bv{\tu} \cdot \del \tilde{\zeta})
 + I_h\zeta_t}{I_h(\Delta \tilde{\zeta})} \\
\innerp{I_h(\bv{\tu} \cdot \del \tilde{\zeta})
 + I_h\zeta_t}{I_h(\tilde{\theta}_x)}
\end{pmatrix}.
\end{align}
Letting $A$ be the matrix on the left-hand side of \eqref{eq:RLS_matrix_form} and $b$ be the vector on the right-hand side of \eqref{eq:RLS_matrix_form}, the parameter update takes the form
\[
\widetilde{\Pr}\vert_{t > t_n} = (A^{-1} b)_1 \vert_{t=t_n}, \qquad \widetilde{\Ra}\vert_{t > t_n}  = \frac{(A^{-1} b)_2}{(A^{-1}b)_1}\Big \vert_{t=t_n}
\]
as long as the matrix $A$ is nonsingular ($\det A \neq 0$ is the nondegeneracy condition in this case), and as long as $(A^{-1}b)_1 \neq 0$. The latter condition only appears as a result of solving for $\Ra$-update from the $\Pr$- and $\Pr\Ra$-updates. The presence of the condition is nevertheless consistent since the Rayleigh-B\'enard system becomes ill-posed if the Prandtl number is close to zero. 

\subsection{Numerical Simulation and Comparison of Algorithms}\label{sect:simulations}

Simulations were performed using  Dedalus version 2.0, which is a ``flexible, open-source, parallelized computational framework for solving general partial differential equations using spectral methods" \cite{burns2020dedalus}. Dedalus features an object-oriented design, symbolic manipulation through a computer algebra system, and reasonable parallelization. More information about the Dedalus codebase is available at \url{http:dedalus-project.org}.

All simulations were performed on a $384 \times 192$ grid, with points having equally spaced horizontal coordinates (Fourier series in the horizontal) and Chebyshev collocation points as vertical coordinates. The domain was $[0,4] \times [0,1]$, meaning that the horizontal dimension was four times as long as the vertical.  All states were initialized using an initial state with $\Ra = 10^5$ and $\Pr = 1$ that had been run out for a sufficiently long time to ensure that all statistics had constant time averages. The initial conditions used for the true system are shown in Figure \ref{fig:initial_state} (following \cite{murri2022}). The assimilating (nudged) system was initialized at a low-mode projection of these states (following \cite{FarhatGlattHoltzMartinezMcQuarrieWhitehead2020}).

\begin{figure}
     \centering
     \begin{subfigure}
         \centering
         \includegraphics[width=\textwidth]{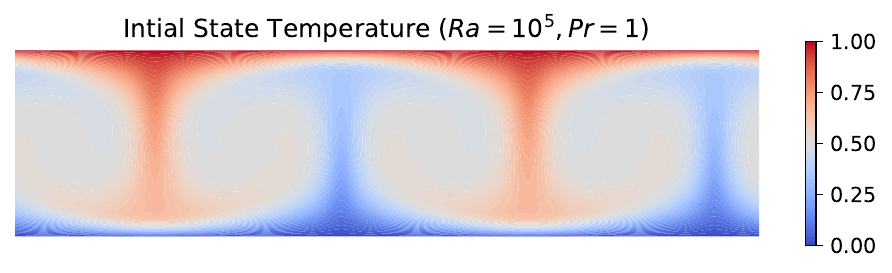}
     \end{subfigure}
     \begin{subfigure}
         \centering
         \includegraphics[width=\textwidth]{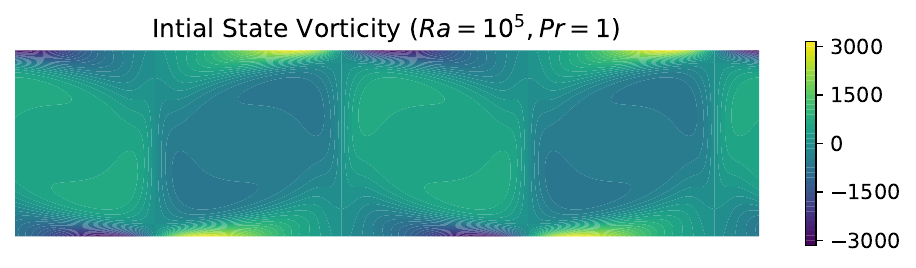}
     \end{subfigure}
     \begin{subfigure}
         \centering
         \includegraphics[width=\textwidth]{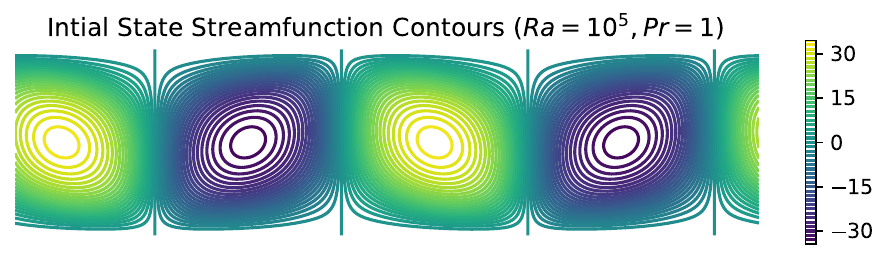}
     \end{subfigure}
\caption{The convective state used as initial condition for computational experiments.  The top is a snapshot of the temperature field, the middle is the vorticity, and the bottom demonstrates contours of the streamfunction.}
\label{fig:initial_state}
\end{figure}

Figure \ref{fig:RB_comparison} compares the performance of $\mathrm{RNI^+}$ and RLS multiparameter algorithms when utilized on the same initial conditions. Both are started at the same initial state, while their assimilating systems are initialized to a low-mode (the lowest 4 modes in each direction) projection of that state. In each case, the true system has $\Pr = 1.0$ and $\Ra = 10^5$, while the assimilating system is initialized with $\widetilde{\Pr} = 1.1$ and $\widetilde{\Ra} = 9 \times 10^4$. In these simulations, the algorithms used a relaxation time (interval between parameter updates) of $0.05$. They used vorticity nudging with $\mu = 8000$ and temperature nudging, also with $\mu_T = 8000$. Note that $\mathrm{RNI^+}$ will fail without temperature nudging whereas RLS will still result in adequate convergence of both the parameters and state even with $\mu_T = 0$, however for this comparison we nudge both the vorticity and temperature.  We also note that for these default parameters, the basic RNI algorithm failed to provide the correct parameter update and led to linearly increasing error in both the state and the parameter estimates.

In Figure \ref{fig:RB_comparison} it is clear that the $\mathrm{RNI^+}$ algorithm converges more quickly both in the state and the parameters; the error from the $\mathrm{RNI^+}$ algorithm seems to bottom out at just after time $t = 0.4$. The RLS algorithm converges more slowly in comparison. However both algorithms appear to reach a final error that is roughly similar. This can be seen in Table \ref{table:PMWvsCHLerror}.

\begin{table}
\centering
\begin{tabular}{c|ccc}
Algorithm & Average State Error & Average Relative $\Ra$ Error & Average Relative $\Pr$ Error \\
\hline
RNI & $2.77 \times 10^{-13}$ & $4.96 \times 10^{-13}$ & $2.65 \times 10^{-11}$ \\
RLS & $1.73 \times 10^{-12}$ & $1.05 \times 10^{-10}$ & $2.66 \times 10^{-12}$
\end{tabular}
\caption{Averaged errors over the interval $t \in [0.8, 1]$ for the RLS and RNI algorithms where the true parameters are $\Ra = 10^5$ and $\Pr = 1$ with a nudging coefficient of $\mu = 8,000$ for both aglorithms.}
\label{table:PMWvsCHLerror}
\end{table}

Table \ref{table:PMWvsCHLerror} shows that while both algorithms converge to a reasonably low state error, the relative parameter error remains larger (at least for the time interval considered here).  Further investigation into this effect for more turbulent settings (higher Rayleigh numbers) will be performed in a later study.

\begin{figure}
    \centering
    \includegraphics[width=\linewidth]{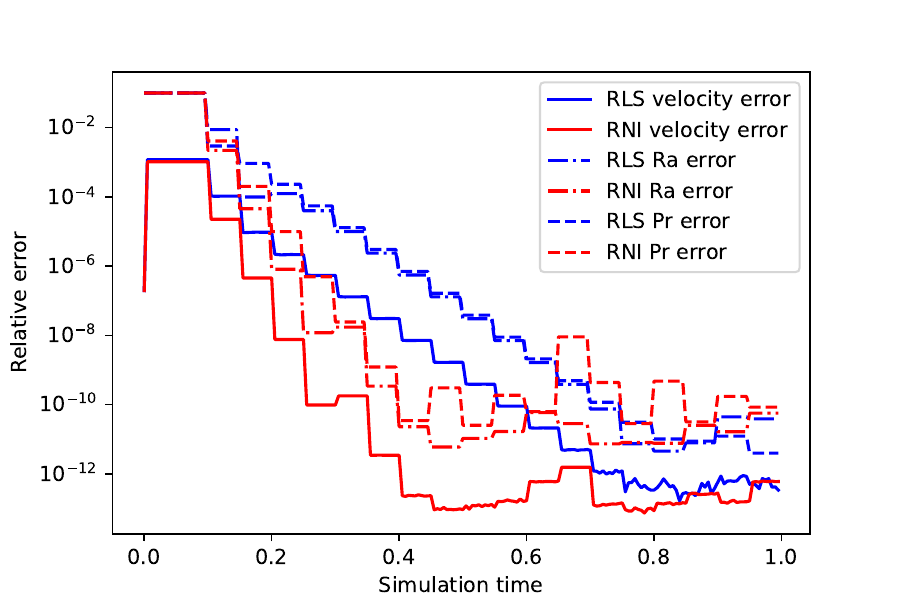}
    \caption{Convergence of the state, and both parameters for both algorithms with a true Rayleigh number of $10^5$ and Prandtl number of $1$ with a nudging coefficient $\mu = 8,000$.}
    \label{fig:RB_comparison}
\end{figure}

The specific choices of the relaxation coefficient $\mu$, interval over which each discrete update is implemented, and all other hyper-parameters do not appear to be critical to the success of these algorithms, at least for $\mathrm{RNI^+}$.  The choice of the relaxation time scale, the value of $\mu$, as well as the initial guess of the unknown parameters, can and does effect the convergence of each algorithm, but slight changes in these hyper-parameters does not appear to drastically alter the results.  We do not fully investigate these hyperparameters here but relegate such an investigation to a later study wherein we will also push the simulations to higher values of the Rayleigh number and hence a more turbulent regime.  Rigorous justification for this setting will also be provided in that future work, relying on rigorous bounds on solutions of the underlying partial differential equations, a level of rigorous analysis that is beyond the scope of the current work.

\section{Conclusions}\label{sect:conc}

This work provides new principled derivations for two algorithms (RNI and RLS) for simultaneous data assimilation and parameter recovery in dissipative dynamical systems. We provide rigorous criteria for convergence of both of these algorithms in the finite dimensional setting, and display the utility of these algorithms on recovery of multiple parameters on the 2-layer Lorenz '96 system. We further numerically investigate the performance of the $\mathrm{RNI^+}$ and RLS algorithms on the 2D Rayleigh-B\'enard convective system and demonstrate that even in this case the unknown parameters and state variables are recovered with reliable accuracy. We note that the RNI algorithm requires modification (which we call $\mathrm{RNI^+}$) to perform well on Rayleigh-Benard convection in the regimes that we tested. 

While the rigorous analysis presented here for the RLS and RNI algorithms is only valid in the finite-dimensional setting, we have demonstrated that numerical approximations to an infinite dimensional PDE will still follow the same results.  In an upcoming work, we will demonstrate that this is generically the case by extending \cref{thm:RLS} to Rayleigh-B\'enard convection, requiring some novel bounds on the underlying PDE as well as the analysis performed above for \cref{thm:RLS} and \cref{thm:RNI}.  Unlike the finite-dimensional case, each PDE must be treated uniquely and so there is no general theorem for the infinite dimensional case.  This future work will include a far more in depth study of the numerical implementation of these algorithms to RBC, including higher resolution simulations in more turbulent regimes (higher Rayleigh numbers).

\section*{Data availability statement}

All data that support the findings of this study are included in the Supplemental Materials to this paper.

\ack{JPW was partially supported by NSF grant DMS-2206762, as well as the Simons Foundation travel grant under 586788. V.R.M. was in part supported by the National Science Foundation through DMS 2213363 and DMS 2206491, as well as the Dolciani Halloran Foundation.}

\section*{References}


\end{document}